\documentclass[12pt]{article}
\usepackage{amsfonts}
\usepackage{amsmath}
\usepackage{amssymb}
\usepackage[mathscr]{eucal}
\usepackage{fullpage}
\usepackage{times}
\usepackage[usenames]{color}
\usepackage[dvipsnames]{xcolor}
\usepackage{hyperref}

\def\eod{\vrule height 6pt width 5pt depth 0pt}
\newenvironment{proof}{\noindent {\bf Proof:} \hspace{.2em}}
                      {\hspace*{\fill}{\eod}}

\newcommand{\floor}[1]{\left\lfloor #1 \right\rfloor}

\newcommand{\pkk}{ \mathrm{pk}}
\newcommand{\vly}{ \mathrm{val}}
\newcommand{\Pkk}{ \mathrm{Peak}}
\newcommand{\Vly}{ \mathrm{Valley}}
\newcommand{\altr}{ \mathrm{altruns}}
\newcommand{\rev}{ \mathrm{Rev}}
\newcommand{\comp}{ \mathrm{Compl}}
\newcommand{\sflip}{ \mathrm{FlipSgn}}
\newcommand{\id}{ \mathsf{id}}

\newcommand{\SR}{ \mathrm{SgnAltrun}}
\newcommand{\SBR}{ \mathrm{SgnBAltrun}}
\newcommand{\SDR}{ \mathrm{SgnDAltrun}}
\newcommand{\SSS}{\mathfrak{S}}
\newcommand{\AAA}{\mathcal{A}}
\newcommand{\BB}{\mathfrak{B}}
\newcommand{\DD}{\mathfrak{D}}
\newcommand{\Alt}{ \mathrm{Alt}}

\newcommand{\inv}{\mathrm{inv}}
\newcommand{\pos}{\mathrm{pos}}
\newcommand{\Negs}{\mathrm{Negs}}
\newcommand{\sign}{\mathrm{sign}}
\newcommand{\nmhalf}{\lfloor (n-1)/2\rfloor}
\newcommand{\ol}[1]{\overline{#1}}
\newcommand{\Snake}{ \mathrm{Snake}}

\newtheorem{theorem}{Theorem}

\newtheorem{corollary}[theorem]{Corollary}

\newtheorem{remark}[theorem]{Remark}

\newtheorem{lemma}[theorem]{Lemma}
\newtheorem{example}[theorem]{Example}

\newcommand{\comment}[1]{}

\newcommand{\magenta}[1]{\textcolor{magenta}{#1}}

\newcommand{\ogreen}[1]{\textcolor{OliveGreen}{#1}}
\newcommand{\blue}[1]{\textcolor{blue}{#1}}

\begin{document}
\title{Signed Alternating-runs enumeration in Classical Weyl Groups}

\author{Hiranya Kishore Dey\\
Department of Mathematics\\
Indian Institute of Technology, Bombay\\
Mumbai 400 076, India.\\
email: hkdey@math.iitb.ac.in
\and
Sivaramakrishnan Sivasubramanian\\
Department of Mathematics\\
Indian Institute of Technology, Bombay\\
Mumbai 400 076, India.\\
email: krishnan@math.iitb.ac.in
}

\maketitle

\section{Introduction}
\label{sec:intro}

For a positive integer $n$, let $[n] = \{1,2,\ldots,n\}$.  Let $\SSS_n$ 
be the symmetric group on $[n]$.  
For $\pi \in \SSS_n$, let $\Pkk(\pi)=\{ 2 \leq i \leq n-1 : 
\pi_{i-1} < \pi_i > \pi_{i+1} \}$ be its set of peaks and 
$\Vly(\pi)=\{ 2 \leq i \leq n-1 : \pi_{i-1} > \pi_i < \pi_{i+1} \}$  
be its set of valleys. Let $\pkk(\pi) = |\Pkk(\pi)|$ and 
$\vly(\pi)= |\Vly(\pi)|.$ 
For an index $i$ with  $2 \leq i \leq n-1$, we say that $\pi$ 
changes direction at $i$ if $i \in \Pkk(\pi) \cup \Vly(\pi)$.
Further, we say 
that $\pi$ has $k$ alternating runs denoted by $\altr(\pi) = k$ if 
$\pi$ changes directions at $k- 1$ indices.  Thus, for 
$\pi \in \SSS_n$, we have %Thus for $\pi \in \SSS_n$, we have 
$\altr(\pi) = \pkk(\pi) + \vly(\pi) + 1$.
For example, the permutation  $\pi=2,3,1,4,6,7,5$ has $2$ peaks, 
$1$ valley and hence $4$ alternating runs. 
Define the alternating runs polynomial in $\SSS_n$ as follows:
\begin{equation}
\label{eqn:altrun_unsigned_over_Sn}
R_n(t)= \sum_{\pi \in \SSS_n}t^{\altr(\pi)} = \sum_{k=1}^{n-1}R_{n,k}t^k.
\end{equation} 

The polynomial $R_n(t)$ is well studied and we refer the reader to 
B\'{o}na's book \cite{bona-combin-permut} for a nice introduction
to alternating runs.   The first papers to study $R_n(t)$ go back
more than a 130 years.  The earliest reference we know for 
the polynomial $R_n(t)$ is by 
Andr{\'e} \cite{Andrealtrun}.
%, who had studied permutations according to their number of alternating runs. 
More recently, $R_n(t)$ has been the subject of 
study in several papers.  See for example Canfield and Wilf
\cite{canfieldwilfalternatingrun}, Chow and Ma
\cite{chowmaaltruntypeb}, Ma
\cite{maaltrun} and Stanley
\cite{stanleyalternatingsequence} and the references
therein.

A simple recurrence exists for $R_n(t)$ and hence  for
$R_{n,k}$, the coefficient of $t^k$ in $R_n(t)$.  However,
a {\it nice formula} for $R_n(t)$ does not seem to be known.
Similarly, several formulae for $R_{n,k}$ exist, but they 
are all somewhat 
complicated and involved a non-trivial summation. 
Stanley in \cite{stanleyalternatingsequence}, Canfield 
and Wilf in \cite{canfieldwilfalternatingrun} and 
later Ma in \cite{maaltrun} present 
explicit formulae for $R_{n,k}$.  We are interested in enumerating 
alternating runs in the alternating group $\AAA_n \subseteq \SSS_n$.  
This leads us to enumerate the 
alternating runs polynomial in $\SSS_n$ with sign taken into account. 
Suprisingly, when enumeration is done with signs, we get
a {\sl compact formula}.
Define the following signed version of \eqref{eqn:altrun_unsigned_over_Sn}:
\begin{equation}
\label{eqn:signed_alt_univ_poly}
\SR_n(t)=\sum_{\pi \in \SSS_n}(-1)^{\inv(\pi)} t^{\altr(\pi)}.
\end{equation}
Our results actually go through with the more general bivariate 
version defined below.
\begin{equation}
\label{eqn:signed_alt_biv_poly}
\SR_n(p,q)=\sum_{\pi \in \SSS_n}(-1)^{\inv(\pi)} p^{\pkk(\pi)}q^{\vly(\pi)}.
\end{equation}

To the best of our knowledge, the polynomial $\SR_n(p,q)$
does not seem to have been 
considered before.  The first main result of this paper 
is the following neat formula. 
 
\begin{theorem}
\label{thm:signed_peak_valley-enumeration}
For positive integers $k$, %\begin{enumerate}
%\item 
let $n=4k$ and $n=4k+1$.  Then,  
$$\SR_{4k}(p,q)=\SR_{4k+1}(p,q)=2(1-p)(1-q)(1-pq)^{2(k-1)}.$$
%\item 
When $n=4k+2$ and $n=4k+3$, we have
$$\SR_{4k+2}(p,q)=\SR_{4k+3}(p,q)=0.$$
%\end{enumerate}
 \end{theorem}

We generalize our results to the classical Weyl Groups.  Let 
$[\pm n] = \{\pm 1, \pm 2,\ldots, \pm n \}$. 
Let $\BB_n$ denote the group of signed permutations 
$\sigma$ on $[\pm n]$ that satisfy 
$\sigma(-i) = -\sigma(i)$ for all $i \in [n]$.   
Let $\DD_n\subseteq \BB_n$ denote the
subset consisting of those elements of $\BB_n$ which have an even
number of negative entries.  
Chow and Ma in \cite{chowmaaltruntypeb} considered the alternating run 
polynomial over $\BB_n$ and Gao and Sun in \cite{Gaosunaltruntyped} 
considered the alternating run polynomial over $\DD_n.$

For $\pi= \pi_1, \pi_2, \dots, \pi_n \in \SSS_n$, let its 
number of inversions be defined by 
$\inv_A(\pi) = |\{ 1 \leq i < j \leq n: \pi_i > \pi_j   \}|$. 
The alternating group, which we denote by $\AAA_n$, is the subgroup 
consisting of the permutations of $\SSS_n$ with an 
even number of inversions.  
Like $\SSS_n$, both $\BB_n$ and $\DD_n$, are Coxeter groups and so 
come with a length function denoted $\inv_B$ and $\inv_D$ 
respectively on its elements.  
The positive parts in both
$\BB_n$ and $\DD_n$ are then defined as the subset containing
elements with the appropriate even length.  

In this work, we enumerate the signed bivariate peak-valley 
polynomials over Type B and Type D Coxeter Groups.  In both these 
cases, we get a simple formula. 
For Type B  and Type D Weyl groups, our main results are 
Theorem \ref{thm:main_theorem_signed_peakvalley_enumerate_over_type_b} and 
Theorem \ref{thm:main_theorem_signed_peakvalley_enumerate_over_type_d} 
respectively. 
All proofs involve sign reversing involutions which identify sets 
outside which complete cancellation occurs.  Enumerating within 
the set is done as the set
has some inductive structure.

Recall that $\AAA_n \subseteq \SSS_n$ is the Alternating group.
For ease of writing statements later, we let
 $\SSS_n^+ = \AAA_n$ and $\SSS_n^- = \SSS_n - \AAA_n$.
Define the following polynomials: 
\begin{equation}
\label{eqn:altrun_unsigned_over_Sn_pusminus}
R_n^+(t) = \sum_{\pi \in \AAA_n}t^{\altr(\pi)} =
\sum_{k=1}^{n-1}R^+_{n,k}t^k 
\hspace{7 mm} \text{and} \hspace{7 mm}
R_n^-(t) = \sum_{\pi \in \SSS_n- \AAA_n}t^{\altr(\pi)} =
\sum_{k=1}^{n-1}R^-_{n,k}t^k.
\end{equation}

It is clear that $R_n^+(t) + R_n^-(t) = R_n(t)$.
Throughout this paper, when we need to refer to both the polynomials 
$R_n^+(t)$ and $R_n^-(t)$, or to the integers $R_{n,k}^+$ or $R_{n,k}^-$,
we use the notation $R_n^{\pm}(t)$ and $R_{n,k}^{\pm}.$
Using Theorem \ref{thm:signed_peak_valley-enumeration}, we trivially get 
a formula for the univariate version defined in 
\eqref{eqn:signed_alt_univ_poly}.  That apart, 
we give several other applications of our signed peak-valley 
polynomial enumeration.

\subsection{Multiplicity of $(1+t)$ as a factor of $R_n^{\pm}(t)$}

Wilf in \cite{wilfaltrun} proved the following result about 
the multiplicity of $(1+t)$ as a divisor of $R_n(t)$.

\begin{theorem}[Wilf]
\label{thm:divisibilty_of_alternatingrunpolynomial_by_highpower_oneplust}
For positive integers $n \geq 4$, the polynomial $R_n(t)$ is divisible by 
$(1 + t)^m$, where $m = \floor{(n -2)/2}$. 
\end{theorem}

Wilf's proof was based on the relation between the Eulerian polynomials and 
$R_n(t)$. Later, B{\'o}na and  
Ehrenborg in \cite{bonaehrenborglogconcavity} gave an inductive 
proof of Theorem 
\ref{thm:divisibilty_of_alternatingrunpolynomial_by_highpower_oneplust}.
The book by  B{\'o}na's \cite{bona-combin-permut} is a 
good reference for this topic. Recently, B{\'o}na in 
\cite{bonaaltrun} gave a proof of Theorem
\ref{thm:divisibilty_of_alternatingrunpolynomial_by_highpower_oneplust}
based on group actions.  Using Theorem 
\ref{thm:signed_peak_valley-enumeration}, we show the 
following.
% refinement of Theorem 
%\ref{thm:divisibilty_of_alternatingrunpolynomial_by_highpower_oneplust}.

\begin{theorem}
\label{thm:Rnplusminust_is_divible_by_1plust_m}
For positive integers $n \geq 4$, let $m = \floor{(n -2)/2}$.
When $n \equiv 0,1$ (mod 4), the 
polynomials $R_n^{\pm}(t)$ are divisible by  $(1 + t)^{m-1}$, 
When $n \equiv 2,3$ (mod 4), the polynomials 
$R_n^{\pm}(t)$ are divisible by $(1 + t)^m$.
\end{theorem} 

Recall that $R_n^+(t) + R_n^-(t) = R_n(t)$.  Thus, 
when $n \equiv 2,3$ (mod 4), Theorem 
\ref{thm:Rnplusminust_is_divible_by_1plust_m} clearly refines Theorem 
\ref{thm:divisibilty_of_alternatingrunpolynomial_by_highpower_oneplust}. 
When $n \equiv 0,1$ (mod 4), 
Theorem \ref{thm:Rnplusminust_is_divible_by_1plust_m} falls short of 
a refinement of Theorem 
\ref{thm:divisibilty_of_alternatingrunpolynomial_by_highpower_oneplust}
as the exponent is lesser by one. 
However, Example \ref{eg_multiplicity-tight}
shows that our result is the best possible in this case.

We also show 
%type B and type D 
counterparts of 
Theorem \ref{thm:Rnplusminust_is_divible_by_1plust_m} 
when alternating runs are summed up over the elements with 
positive sign in classical Weyl groups.  
Zhao in \cite{zhaothesis} and Gao and Sun 
in \cite{Gaosunaltruntyped}, gave type B and D refinements 
based on the sign of the first letter.
From these refinements, a type B and type D counterpart of Theorem
\ref{thm:divisibilty_of_alternatingrunpolynomial_by_highpower_oneplust}
can be inferred.  
We prove 
signed counterparts of Theorem 
\ref{thm:Rnplusminust_is_divible_by_1plust_m} to
$\BB_n^{\pm}$ and $\DD_n^{\pm}$.  These are 
presented in Theorem 
\ref{thm:divisibiltyforplusminuspolynomialstypeB} and 
Theorem \ref{thm:divisibiltyfortyped}.  Thus, we refine
in a different sense, type B and type D counterparts of
Theorem 
\ref{thm:Rnplusminust_is_divible_by_1plust_m}.

We also have signed refinements of the refinements of 
Zhao %in \cite{zhaothesis} 
and Gao and Sun.  These fit snugly into our type B and type D 
counterpart of the group action based proof of
Theorem 
\ref{thm:divisibilty_of_alternatingrunpolynomial_by_highpower_oneplust}
and a refinement of 
Theorem \ref{thm:Rnplusminust_is_divible_by_1plust_m}.  We 
refer the reader to 
\cite[Theorems 11,17]{dey-siva-grp-action-bona-wilf} 
for this refinement.

%refinements.  \blue{Our signed refinements are Theorem
%\ref{thm:dbl_rstrct_div_b_even} and Theorem 
%\ref{typed}.}  Corollaries of these are 

\subsection{Refining moment-like identities involving $R^{\pm}_{n,k}$}
Moment-like identities involving the $R_{n,k}$'s were given by Comtet 
\cite{comtetadvanced} as an exercise.
Chow and Ma in \cite{chowmaaltruntypeb} deduced these identities from 
Theorem 
\ref{thm:divisibilty_of_alternatingrunpolynomial_by_highpower_oneplust}.
The result is as follows.

\begin{lemma}
\label{lem:moment-identity-chow-ma}
Let $n,k$ be positive integers with $n \geq 2k+4$,  then
\begin{equation}\label{eqn:momenttypeidentitychowma}
1^kR_{n,1}+3^kR_{n,3}+5^kR_{n,5}+\cdots  =  
2^kR_{n,2}+4^kR_{n,4}+6^kR_{n,6}+\cdots.
\end{equation}
\end{lemma}

Using Theorem \ref{thm:Rnplusminust_is_divible_by_1plust_m}, 
we have the following counterpart of the above result involving
$R_{n,k}^{\pm}$.

%\red{Theorem \ref{thm:Rnplusminust_is_divible_by_1plust_m} has
%two parts (0,1 and 2,3 mod 4).  Using it, we should also get
%two parts na?}

\begin{theorem}
\label{thm:sum_of_kth_powers_of_odd_equals_sum_of_kth_powers_of_even}
For positive integers $n \equiv 0,1$ (mod 4) and positive 
integers $k$ such that $n \geq 2k+6$, we have 
\begin{equation}
1^kR^{\pm}_{n,1}+3^kR^{\pm}_{n,3}+5^kR^{\pm}_{n,5}+\cdots  =  2^kR^{\pm}_{n,2}+4^kR^{\pm}_{n,4}+6^kR^{\pm}_{n,6}+\cdots. \label{eqn:sum_of_kth_powers_of_odd_equals_sum_of_kth_powers_of_even_even_01mod4}
\end{equation}
For positive integers $n \equiv 2,3$ (mod 4) and positive
integers $k$ such that $n \geq 2k+4$, we have 
\begin{equation}
1^kR^{\pm}_{n,1}+3^kR^{\pm}_{n,3}+5^kR^{\pm}_{n,5}+\cdots =  2^kR^{\pm}_{n,2}+4^kR^{\pm}_{n,4}+6^kR^{\pm}_{n,6}+\cdots. \label{eqn:sum_of_kth_powers_of_odd_equals_sum_of_kth_powers_of_even_even}
\end{equation}
\comment{
For $n \geq 2k+6$, we have 
\begin{equation}
\label{eqn:sum_of_kth_powers_of_odd_equals_sum_of_kth_powers_of_even_even}
1^kR^{\pm}_{n,1}+3^kR^{\pm}_{n,3}+5^kR^{\pm}_{n,5}+\cdots  
=  
2^kR^{\pm}_{n,2}+4^kR^{\pm}_{n,4}+6^kR^{\pm}_{n,6}+\cdots.
%	 \\
%	1^kR^-_{n,1}+3^kR^-_{n,3}+5^kR^-_{n,5}+\cdots & = &  2^kR^-_{n,2}+4^kR^-_{n,4}+6^kR^-_{n,6}+\cdots.
%	\label{eqn:sum_of_kth_powers_of_odd_equals_sum_of_kth_powers_of_even_odd}
\end{equation}}
\end{theorem} 

For positive elements in Type B and Type D Weyl groups, we get similar
identities.  These are 
Theorem 
\ref{thm:sum_of_kth_powers_of_odd_equals_sum_of_kth_powers_signed_type_b_evenodd} 
and  Theorem 
\ref{thm:moment-ids-typed-pm}.

\subsection{Alternating permutations in $\SSS_n^{\pm}$, 
$\BB_n^{\pm}$ and $\DD_n^{\pm}$}
A permutation $\pi = \pi_1, \pi_2, \cdots, \pi_n \in \SSS_n$ is 
said to be alternating if $\pi_1 > \pi_2 < \pi_3 > \pi_4 < \cdots.$ 
That is, $\pi_i < \pi_{i+1}$ for even indices $i$ and $\pi_i > \pi_{i+1}$ 
for odd indices $i$. Similarly $\pi$ is said to be reverse alternating 
if $\pi_1 < \pi_2 > \pi_3 < \pi_4 > \dots $. 
Let $\Alt_n \subseteq \SSS_n$ denote the set of alternating 
permutations and let $E_n= |\Alt_n|$.  The following 
is a well-known  result of  Andr{\'e} \cite{Andresecx}.
  
\begin{theorem}[Andr{\'e}]
\label{thm:egf_for_alt_perm_symmetricgroup}
Let $E_n$ be the number of alternating permutations in $\SSS_n$.  The
exponential generating function of $E_n$ is as follows.
\begin{equation}
\label{eqn:egf_for_alt_perm_symmetricgroup}
\sum_{n=0}^{\infty} {E_n}\frac{x^n}{n!}	
=  \sec x+ \tan x. 
\end{equation}
\end{theorem}

%Andr{\'e} 
%in \cite{Andresecx} gave the exponential generating function 
%for alternating permutations.  
We refer the reader to Stanleys  survey \cite{stanley-survey-alt-perms} 
for a comprehensive coverage of this topic.
Define $E_n^+= |\Alt_n \cap \AAA_n|$ and 
$E_n^-= |\Alt_n \cap (\SSS_n -\AAA_n)|.$ 
When $n=0$, we set $E_{0}^+=1$ and $E_{0}^-=0$. 
Using our proof of Theorem \ref{thm:signed_peak_valley-enumeration},  
we show the following.
%exponential generating 
%function of $E_n^{\pm}$ as follows.
  
\begin{theorem}
\label{thm:egf_for_alt_perm_even_and_odd}
The exponential generating functions of $E_n^{\pm}$ are as follows.
\begin{equation}
\label{eqn:egf_for_alt_perm_symmetricgroupplus}
\sum_{n=0}^{\infty} {E_n^ {\pm}}\frac{x^n}{n!}	
=  \frac{1}{2} \bigg(\sec x+ \tan x \pm \cos x \pm x\bigg).
\end{equation}
\end{theorem}

Steingr{\'i}msson in  \cite{steingrimsson-indexed-perms} 
extended the results for colored permutations.
We give type B and D counterparts of Theorem 
\ref{thm:egf_for_alt_perm_even_and_odd} which refine
Steingr{\'i}msson's results.
For Type B and Type D 
Weyl groups, our results are Theorem 
\ref{thm:egf_for_alt_perm_hyperoctahedralgroupplusandtyped} and 
Theorem \ref{thm:egf_for_alt_perm_demihyperoctahedralgroupplusminus}. 

\subsection{Snakes in $\BB_n^{\pm}$ and $\DD_n^{\pm}$}
Snakes of type B were were introduced by Springer 
in \cite{springer-remarks-combin} and are closely related to 
type B alternating permutations.  Their exponential generating 
functions involve well known trigonometric functions.  
Gao and Sun in \cite{Gaosunaltruntyped} 
refined the enumeration of Snakes over $\DD_n$ and $\BB_n - \DD_n$.
Our work gives an alternate proof of their result and 
further refines the enumeration of types $\BB_n^{\pm}$,
$\DD_n^{\pm}$ and $(\BB_n-\DD_n)^{\pm}$ snakes.  Our
main results are Theorem \ref{thm:egfforsnakesb}
and Theorem	\ref{thm:egfforsnakesd}.

%\subsection{Central Limit Theorems for $R_n^{\pm}(t)$}
%\red{\bf No mention of Central Limit Theorems?}

\section{Preliminaries}
\label{sec:prelim}

We divide $\SSS_n$ into four disjoint subsets and compute the
signed alternating runs polynomial on these subsets. 
We first partition $\SSS_n$ based on the type of the first and 
the last pairs.  Either pair could be an ascent or a
descent.  When $n \geq 3$, we get the following four sets:

\begin{eqnarray*}
	 \SSS_{n,a,a} & = & \{\pi \in \SSS_n : \pi_1 < \pi_2, \pi_{n-1}< \pi_n \},  \hspace{5 mm}
	 \SSS_{n,a,d}= \{\pi \in \SSS_n : \pi_1 < \pi_2, \pi_{n-1}> \pi_n \}, \\
	 \SSS_{n,d,a} & = &  \{\pi \in \SSS_n : \pi_1 > \pi_2, \pi_{n-1}< \pi_n \},   \hspace{5 mm}
	 \SSS_{n,d,d}= \{\pi \in \SSS_n : \pi_1 > \pi_2, \pi_{n-1}> \pi_n \}. 
\end{eqnarray*}

We will enumerate the contribution to $\SR_n(p,q)$ from each 
of the four sets. Hence, we define the following four 
alternating runs enumerator polynomials:

\begin{enumerate}
	\item $\SR_{n,a,a}(p,q)=\sum_{\pi \in \SSS_{n,a,a}}(-1)^{\inv(\pi)}p^{\pkk(\pi)}q^{\vly(\pi)} ,$
	\item $\SR_{n,a,d}(p,q)=\sum_{\pi \in \SSS_{n,a,d}}(-1)^{\inv(\pi)}p^{\pkk(\pi)}q^{\vly(\pi)} ,$
	\item $\SR_{n,d,a}(p,q)=\sum_{\pi \in \SSS_{n,d,a}}(-1)^{\inv(\pi)}p^{\pkk(\pi)}q^{\vly(\pi)} ,$
	\item $\SR_{n,d,d}(p,q)=\sum_{\pi \in \SSS_{n,d,d}}(-1)^{\inv(\pi)}p^{\pkk(\pi)}q^{\vly(\pi)} .$ 
\end{enumerate}

We also define the following sets to keep track of the case when
$\pi \in \AAA_{n}$. %or $\pi \in (\SSS_n- \AAA_{n})$. 
Define
\begin{eqnarray*}
	\AAA_{n,a,a} & = & \{\pi \in \AAA_n : \pi_1 < \pi_2, \pi_{n-1}< \pi_n \},  \hspace{5 mm}
	\AAA_{n,a,d}= \{\pi \in \AAA_n : \pi_1 < \pi_2, \pi_{n-1}> \pi_n \}, \\
	\AAA_{n,d,a} & = &  \{\pi \in \AAA_n : \pi_1 > \pi_2, \pi_{n-1}< \pi_n \},   \hspace{5 mm}
	\AAA_{n,d,d}= \{\pi \in \AAA_n : \pi_1 > \pi_2, \pi_{n-1}> \pi_n \}. 
\end{eqnarray*}

It is easy to see the following relation between $\pkk(\pi)$ and 
$\vly(\pi)$ depending on which of the four sets $\pi$ lies in.  
When $\pi \in \SSS_{n,a,a}$ and when
$\pi \in \SSS_{n,d,d}$, we have $\pkk(\pi) = \vly(\pi)$.
On the other hand, when  $\pi \in \SSS_{n,a,d}$, we have
$\pkk(\pi) = 1 + \vly(\pi)$  and when
$\pi \in \SSS_{n,d,a}$, we have $\vly(\pi) = \pkk(\pi) + 1$. 
Thus, for all $\pi$, we have $|\pkk(\pi) - \vly(\pi)| \leq 1$.

For $\pi \in \SSS_n,$ let $\pi'$ be obtained from $\pi$ by deleting
the letter $n$.  Equivalently, $\pi'$ can be obtained by restricting
$\pi$ to the set $[n-1]$.  Suppose $\pi \in \AAA_n$.  Then, it is easy 
to see that both $\pi' \in \AAA_{n-1}$ and $\pi' \in \SSS_{n-1} - \AAA_{n-1}$ 
are possible.  An identical statement is true  when we get $\pi'$ 
from $\pi \in \SSS_n - \AAA_n$.

We will use induction and hence add the letter $n$ to permutations 
in $\SSS_{n-1}$. We need to keep track 
of the position of the letter $n$ in $\pi$ as we want to take 
the sign of the permutations into account.   
Let $\pi= \pi_1, \pi_2, \ldots, \pi_{n-1} \in \SSS_{n-1}$.
We term the left-most position before $\pi_1$ as the initial 
or the $0$-th gap and the right-most position after $\pi_{n-1}$ 
as the final or $(n-1)$-th gap.
For  $1 \leq i \leq n-2$, we denote the gap between $\pi_i$ and 
$\pi_{i+1}$ as the $i$-th gap. 
%and recall $\pi'$ is  $\pi$ restricted to $[n-1]$.
For $\pi' \in \SSS_{n-1}$, we call the $i$-th gap of $\pi'$ to 
be even if $n-1-i$ is even and the $i$-th gap of $\pi'$ to be 
odd if $n-1-i$ is odd.  The following remark about the number of
odd and even gaps based on the parity of $n$ is easy to
see.

\begin{remark}
\label{rem:parity _of_gaps}
If $\pi' \in \SSS_{n-1}$ with $n-1$ being even, the total 
number of gaps is $n$ which is odd.  One can check that the 
number of even gaps is $\floor{n/2}+1 $ and the number of 
odd gaps is $\floor{n/2}$.  

On the other hand, if $\pi' \in \SSS_{n-1}$ with $n-1$ being
odd, the total number of gaps is $n$ which is even. Here, we
will have $n/2$ even and odd gaps. 
\end{remark}

By definition, since the gap set is $\{0,1,\ldots,n-1\}$, we 
define $[n-1]_{0} = \{0,1,2,\dots,n-1\}$.

\subsection{Two easy bijections and some properties} 

In this Section, we mention some easy bijections which we 
will use several times in this work. 
Let, $\pi= \pi_1, \pi_2, \dots, \pi_n \in \SSS_n$. Define the
following bijections:
%\begin{enumerate}
%	\item  
Define the  {\sl complementation bijection} 
$\comp: \SSS_n \mapsto \SSS_n$ by 
$$\comp(\pi_1,\pi_2,\dots,\pi_n)= n+1-\pi_1, n+1-\pi_2, \dots, n+1- \pi_n. $$
	%\item  
Define the {\sl reverse bijection} $\rev: \SSS_n \mapsto \SSS_n$ 
by $$\rev(\pi_1,\pi_2,\dots,\pi_n)= \pi_n, \pi_{n-1}, \dots, \pi_1. $$
%\end{enumerate}

\begin{lemma}
\label{lem:properties_of_complementation}
The bijection $\comp$ has the following properties:
\begin{enumerate}
	\item[P1:]
\label{itm:prop1ofC} 
For $\pi \in \SSS_n$, we 
have $\Vly(\pi)= \Pkk(\comp(\pi))$ and $\Pkk(\pi)= \Vly(\comp(\pi))$.
Thus, $\vly(\pi)= \pkk(\comp(\pi))$ and $\pkk(\pi)= \vly(\comp(\pi))$.

\item[P2:] 
	\label{itm:prop2ofC} Let $\pi \in \AAA_n$. Then, 
$\comp(\pi) \in \AAA_n$	if and only if $n \equiv 0,1$ (mod 4). 
Similarly, if $\pi \in \SSS_n - \AAA_n,$ then $\comp(\pi) 
	\in \SSS_n-\AAA_n$ if and only if $n \equiv 0,1$ (mod 4).
\end{enumerate}
\end{lemma}

\begin{proof}
It is easy to see that the peak and the valley sets switch when
we complement the permutation $\pi$. 
%has $k$ peaks and $l$ valleys, 
%then $C(\pi)$ has $l$ peaks and $k$ valleys. 
Thus, P1 is proved.
For the second part, we note that $\inv_A(\pi) + \inv_A(\comp(\pi)) 
= \binom{n}{2}$. As $\binom{n}{2}$ is even if and only if 
$n \equiv 0,1$ (mod 4), $\inv_A(\pi)$ and $\inv_A(\comp(\pi))$ have
the same parity if and only if $n \equiv 0,1$ (mod 4). 
This completes the proof of P2. 
\end{proof}

\begin{lemma}
\label{lem:01mod4basic}
For positive integers $n \equiv 0,1$ (mod 4), the following hold.
\begin{eqnarray}
\label{eqn:flip-pk-val-01-mod-4-ad}
	\SR_{n,a,d}(p,q) & = & \SR_{n,d,a}(q,p),  \\
%\blue{Why is the order $(q,p)$ on the RHS?}
\label{eqn:flip-pk-val-01-mod-4-aa}
	\SR_{n,a,a}(p,q) & = & \SR_{n,d,d}(p,q).
\end{eqnarray}
\end{lemma}

\begin{proof}
We first consider \eqref{eqn:flip-pk-val-01-mod-4-ad}.
When $n \equiv 0,1$ (mod 4), 
by Lemma \ref{lem:properties_of_complementation}, 
$\comp$ is a 
map from $\AAA_n$ to $\AAA_n$. Further, it is clear that
$\comp$ is 
a bijection from $\AAA_{n,a,d}$ to $ \AAA_{n,d,a}$ that 
swaps the peak set and the valley set.  It is also easy to
see that if $\pi \in \SSS_{n,a,d}$ then 
$\pkk(\pi) = \vly(\pi)+1$.  
This proves 
\eqref{eqn:flip-pk-val-01-mod-4-ad}.

The proof of \eqref{eqn:flip-pk-val-01-mod-4-aa}
is similar.  We note that if $\pi \in \SSS_{n,a,a} \cup 
\SSS_{n,d,d}$, then, $\pkk(\pi) = \vly(\pi)$.  
Thus, there is no need to swap the variables $p$ and $q$.  
This completes the proof.  
\end{proof}

We next move onto similar properties of the reversing bijection.

\begin{lemma}
\label{lem:properties_of_reverse}
The bijection $\rev$ has the following properties.
\begin{enumerate}
%\item \label{itm:prop1ofR} 
\item[R1:] 
For $\pi \in \SSS_n$, we have 
$\vly(\pi)= \vly(\rev(\pi))$ and $\pkk(\pi)= \pkk (\rev(\pi))$.
		
%\item \label{itm:prop2ofR} 
\item[R2:] 
Let $\pi \in \AAA_n$. Then, 
$\rev(\pi) \in \AAA_n$ if and only if $n \equiv 0,1$ (mod 4).
	\end{enumerate}
\end{lemma}

As the proof of Lemma \ref{lem:properties_of_reverse} is similar to 
the proof of Lemma \ref{lem:properties_of_complementation}, 
we omit its proof. Using this map $\rev$, we directly get the following.

\begin{lemma}
\label{lem:23mod4}
For positive integers $n \equiv 2,3$ (mod 4), the following hold.
\begin{eqnarray}
\label{eqn:23_mod4_ad_zero}
\SR_{n,a,d}(p,q)& = & 0 \mbox { and } \SR_{n,d,a}(p,q)=0, \\
\label{eqn:23_mod4_aa_neg}
\SR_{n,a,a}(p,q)&  =& -\SR_{n,d,d}(p,q), \\
\label{eqn:23_mod4_total_zero}
\SR_{n}(p,q) & = & 0. 
\end{eqnarray}
\end{lemma}

\begin{proof}
We first consider \eqref{eqn:23_mod4_ad_zero}.
When $n \equiv 2,3$ (mod 4), by Lemma 
\ref{lem:properties_of_reverse}, 
$\rev$ is a bijection from $\AAA_n$ to $\SSS_n - \AAA_n$. 
Further, by Lemma \ref{lem:properties_of_reverse}, $\rev$ 
is a bijection from $\AAA_{n,a,d}$ to $\SSS_{n,a,d} - \AAA_{n,a,d}$
which preserves the number of peaks and valleys.  
A similar argument %except that $\rev$ is a bijection 
%from $\AAA_{n,a,a}$ to $\SSS_{n,d,d} - \AAA_{n,d,d}$ 
shows \eqref{eqn:23_mod4_aa_neg}.  
Summing \eqref{eqn:23_mod4_ad_zero} and \eqref{eqn:23_mod4_aa_neg}
gives \eqref{eqn:23_mod4_total_zero}.
\end{proof}

\begin{remark}
\label{rem:evennessofcoefficientsofRnt}
Consider the bijection $\comp:\SSS_n \mapsto \SSS_n$.  It is 
easy to note that
$\pi$ and $\comp(\pi)$ have the same number of alternating runs. 
Thus, when $n\geq 2$, each coefficient of $R_n(t)$ is even. 
\end{remark}

\begin{lemma}
	\label{lem:inserting_n_beforeandafter_peak}
Let $\pi'= \pi_1, \pi_2, \dots, \pi_{n-1} \in \SSS_{n-1}$ with
$\Pkk(\pi') = k_1, k_2, \dots, k_r$.  Then, for $1 \leq i \leq r$, 
the two following two permutations 
$\psi_1 =  \pi_1, \pi_2, \dots, \pi_{k_i-1}, n, \pi_{k_i}, \dots,\pi_{n-1}$ 
and 
$\psi_2 =  \pi_1, \pi_2, \dots, \pi_{k_i-1}, \pi_{k_i}, n,  \dots,\pi_{n-1}$ 
have the same number of peaks and valleys but have different signs. 
\end{lemma}

\begin{proof}
It is clear when $n$ is inserted before and after the 
$k_i$-th index,  that $\psi_1 $ and $\psi_2$ have the same number of 
peaks and valleys.  
Further, we get $\inv (\psi_1)= \inv(\pi')+n-k_i$ 
and $\inv (\psi_2)= \inv(\pi')+n-k_i-1.$  Thus, 
$\inv (\psi_1)$ and $\inv (\psi_2)$ have different parity. 
This completes the proof. 
\end{proof}

\section{Recurrences for the Signed Bivariate Polynomials}

We wish to prove recurrences for the polynomials 
$\SR_{n-1,a,d}(p,q)$, $\SR_{n-1,d,a}(p,q)$, $\SR_{n-1,a,a}(p,q)$ 
and $\SR_{n-1,d,d}(p,q)$. We start with the following remark 
which says that inserting $n$ cannot change the type of both
the starting and the ending run.

\begin{remark}
\label{rem:adding_n_cannot_flip_both_end}
Inserting $n$ at any gap in  $\pi' \in \SSS_{n-1,d,a}$ 
cannot give $\pi \in \SSS_{n,a,d}$ and likewise, 
inserting $n$ at any gap in  $\pi' \in \SSS_{n-1,a,d}$ 
cannot give $\pi \in \SSS_{n,d,a}$. 
Similarly, inserting $n$ at any position in $\pi' \in 
\SSS_{n-1,a,a}$ cannot give $\pi \in \SSS_{n,d,d}$ 
and inserting $n$ at any position in $\pi' \in \SSS_{n-1,d,d}$ 
cannot give $\pi \in \SSS_{n,a,a}$ . 
\end{remark}

Our recurrence relation for the polynomials depends on the parity of $n$ and so we bifurcate
the remaining part into two cases.

\subsection{When $n$ is odd}
\begin{theorem}
\label{thm:ad_equals_da_for_nminusone}
For odd positive integers $n$, we have $$q\SR_{n-1,a,d}(p,q)=p\SR_{n-1,d,a}(p,q).$$
\end{theorem}

\begin{proof}
If $n$ is odd, $n-1$ is either $0 \mbox{ or } 2$ (mod 4). 
If $n-1 \equiv 0$ (mod 4), consider the bijection $\comp$ from  
$\SSS_{n-1,a,d}$ to $\SSS_{n-1,d,a}$.  
We have $\pkk(\pi) = \vly(\comp(\pi))$ and $\vly(\pi) = 
\pkk(\comp(\pi))$.  Moreover, if
$\pi \in \SSS_{n-1,a,d}$, then  %has $1$ more peak than valley. 
%Thus, 
$\pkk(\pi)=\vly(\pi)+1=\pkk(\comp(\pi))+1$. Similarly, we get
$\vly(\pi)=\pkk(\pi)-1=\vly(\comp(\pi))-1$. 
Thus, if $\pi \in \SSS_{n-1,a,d}$,  then,
	\begin{equation}
	\label{eqn:pkvlyrelation}
	q \times p^{\pkk(\pi)}q^{\vly(\pi)} = p \times p^{\pkk(\comp(\pi))}
q^{\vly(\comp(\pi))}.
	\end{equation}
As equation \eqref{eqn:pkvlyrelation} is true for all
$\pi \in \SSS_{n-1,a,d}$, summing over elements of $\SSS_{n-1,a,d}$
completes the proof. 
	
If  $n-1 \equiv 2$ (mod 4), then both $\SR_{n-1,a,d}(p,q)=0$ and 
$\SR_{n-1,d,a}(p,q)=0$ by the map $\rev$. The proof is complete. 
\end{proof}

From Lemma \ref{lem:01mod4basic} and Lemma \ref{lem:23mod4}, we know
that instead of getting a recurrence for the four polynomials, it
suffices to get a recurrence for just the two polynomials 
$\SR_{n,a,d}(p,q)$ and 
$\SR_{n,a,a}(p,q)$.

\begin{theorem}
\label{thm:peakvalleyrecurrencesfornodd}
For odd positive integers $n \geq 3$, the following recurrence 
relations hold:
	\begin{eqnarray}
	\SR_{n,a,d}(p,q) &  = & -pq\SR_{n-1,a,d}(p,q)  -p\SR_{n-1,a,a}(p,q) \nonumber 
	\\ & & -p\SR_{n-1,d,d}(p,q),  \label{eqn:nad_n_odd} \\
%	\SR_{n,d,a}(p,q) &  = & (2+pq)\SR_{n-1,d,a}(p,q)  +q\SR_{n-1,a,a}(p,q) \nonumber
%	\\ & & +q\SR_{n-1,d,d}(p,q), \label{eqn:nda_n_odd} \\ 
	\SR_{n,a,a}(p,q) & = & q\SR_{n-1,a,d}(p,q) - 
	p\SR_{n-1,d,a}(p,q) \nonumber 
	\\ & & +\SR_{n-1,a,a}(p,q)\label{eqn:naa_n_odd}  \\
	&= &   \SR_{n-1,a,a}(p,q) \label{eqn:naa_n_odd_simplified} 
%	\SR_{n,d,d}(p,q) & = & q\SR_{n-1,a,d}(p,q) - 
%	p\SR_{n-1,d,a}(p,q) \nonumber
%	\\
%	& & + \SR_{n-1,d,d}(p,q) \label{eqn:ndd_n_odd} \\
%	& = &  \SR_{n-1,d,d}(p,q) \label{eqn:ndd_n_odd_simplified}
	\end{eqnarray}
\end{theorem}

\begin{proof}
We first prove \eqref{eqn:nad_n_odd}. 
%For a permutation $\pi \in \SSS_n$, let $\pi'$ denote the restriction on the $n-1$ letters ${1,2,\ldots,n-1}$. It is easy to see that by inserting $n$ into various positions of all the permutations in $\SSS_{n-1}$, we get $\SSS_n$. 
By Remark \ref{rem:adding_n_cannot_flip_both_end}, only 
the following three polynomials contribute to $\SR_{n,a,d}(p,q)$. 
\begin{eqnarray}
\SR_{n,a,d}(p,q) & = & \hspace{-3 mm} 
\sum_{\pi \in \SSS_{n,a,d}, \pi' \in \SSS_{n-1,a,d}} 
(-1)^{\inv(\pi)} p^{\pkk(\pi)}q^{\vly(\pi)} \nonumber \\ 
& & + \sum_{\pi \in \SSS_{n,a,d}, \pi' \in \SSS_{n-1,a,a}} 
(-1)^{\inv(\pi)}p^{\pkk(\pi)}q^{\vly(\pi)} \nonumber \\ 
& & + \sum_{\pi \in \SSS_{n,a,d}, \pi' \in \SSS_{n-1,d,d}} 
(-1)^{\inv(\pi)} p^{\pkk(\pi)}q^{\vly(\pi)} \nonumber 
\end{eqnarray}
We consider the contribution from each of the three terms 
to $\SR_{n,a,d}(p,q)$ separately.

\begin{enumerate}
\item Suppose $\pi' \in \SSS_{n-1,a,d}$ and we want to  get 
all $\pi \in \SSS_{n,a,d}$ by inserting $n$ in 
different gaps of $\pi'$. 
We claim that we cannot insert $n$ in the initial or final gap, 
as by such an insertion, we will not get $\pi \in \SSS_{n,a,d}$. 
By inserting $n$ in 
all other gaps of  $\pi',$ we will get $\pi \in \SSS_{n,a,d}$. 
Let $k_1, k_2, \dots,  k_r$ be the peaks of $\pi'$. 
By Lemma \ref{lem:inserting_n_beforeandafter_peak}, inserting $n$ 
at the $k_i-1$ and  $k_i$-th gap gives permutations with equal 
number of peaks and valleys but with different sign. 
Therefore, inserting $n$ 
at the gaps $k_1-1, k_1, k_2-1, k_2 \dots,  k_r-1, k_r$ 
contributes $0$ in total. 
		
The remaining gaps are $A = [n-1]_{0} - \{0,k_1-1, k_1, k_2-1, k_2, 
\dots,  k_r-1, k_r,n-1\}$.  
It is easy to see that inserting $n$ at any of gaps from the set $A$
increases both the number of peaks and valleys by $1$. Moreover, by 
Remark \ref{rem:parity _of_gaps}, the total number of even and 
odd gaps are $\floor{n/2}+1$ and $\floor{n/2}$ respectively.  
As $n-1$ is even, both $0$ and $n-1$ are even gaps.  
Hence, the number of even gaps available in the set $A$ 
is $ \floor{n/2}+1 - (r+2)$. Similarly, the number of 
odd gaps available in the set $A$ 
is $\floor{n/2}- r$.

Thus, $A$ has one more odd gap than even gap.  %such that inserting $n$ 
%results in the increase of peaks and valleys both by $1$ 
%is $\floor{n/2}- r$. 
Hence, the contribution from $\pi' \in \SSS_{n-1,a,d} $ to  
$\SR_{n,a,d}(p,q)$ is $(-pq) \times (-1)^{\inv(\pi')}p^{\pkk(\pi')}q^{\vly(\pi')}$.

\item To get  $\pi \in \SSS_{n,a,d}$ from 
$\pi' \in \SSS_{n-1,a,a},$ there is only one choice:
insert $n$ at the $(n-2)$-nd gap. 
By inserting $n$ at the $(n-2)$-nd gap in $\pi'$, 
the number of peaks will increase by $1$.  
Further, the number of inversions will also increase by $1$. 
As $\inv(\pi)=\inv(\pi')+1$,  $\pi$ and $\pi'$ 
are of opposite sign. 
Thus, the contribution 
from $\pi' \in \SSS_{n-1,a,a} $ to  $\SR_{n,a,d}(p,q)$ 
is $(-p) \times (-1)^{\inv(\pi')}p^{\pkk(\pi')}q^{\vly(\pi')}$. 
		
\item To get $\pi \in \SSS_{n,a,d}$ from  
$\pi' \in \SSS_{n-1,d,d},$ there is again only one choice:
insert $n$ at the first gap. 
By this, the number of peaks will increase by $1$.  Further, 
the number of inversions will increase by $n-2$ which is odd.
Hence $\pi$ and $\pi'$ are of opposite sign.  Thus, 
the contribution from $\pi' \in \SSS_{n-1,d,d} $ to 
$\SR_{n,a,d}(p,q)$ is $(-p) \times (-1)^{\inv(\pi')}p^{\pkk(\pi')}
q^{\vly(\pi')}$.
\end{enumerate}
	
With these, the proof of \eqref{eqn:nad_n_odd} is complete. 
We now consider  \eqref{eqn:naa_n_odd}.

By Remark \ref{rem:adding_n_cannot_flip_both_end}, 
$\SR_{n,a,a}(p,q)$ can be written as a sum of the following three 
polynomials.  
	
\begin{eqnarray}
\SR_{n,a,a}(p,q) & = & \hspace{-3 mm} \sum_{\pi \in \SSS_{n,a,a}, \pi' \in \SSS_{n-1,a,a}} 
(-1)^{\inv(\pi)} p^{\pkk(\pi)} q^{\vly(\pi)} \nonumber \\ 
& & + \sum_{\pi \in \SSS_{n,a,a}, \pi' \in \SSS_{n-1,a,d}} 
(-1)^{\inv(\pi)} p^{\pkk(\pi)} q^{\vly(\pi)} \nonumber \\ 
& & + \sum_{\pi \in \SSS_{n,a,a}, \pi' \in \SSS_{n-1,d,a}} 
(-1)^{\inv(\pi)} p^{\pkk(\pi)} q^{\vly(\pi)} \nonumber 
\end{eqnarray}

We again consider the contribution to $\SR_{n,a,a}(p,q)$ from each of the three terms. 
\begin{enumerate}
\item Suppose $\pi' \in \SSS_{n-1,a,a}$ and we want to get all 
$\pi \in \SSS_{n,a,a}$ by inserting $n$ in different gaps of $\pi'$. We clearly 
cannot insert $n$ in the initial or the penultimate gap, as then the resulting 
$\pi \not\in \SSS_{n,a,a}$. Inserting $n$ at the final gap of $\pi' \in 
\SSS_{n-1,a,a}$ does not change the number of peaks and valleys. 
Let $k_1, k_2, \dots,  k_r$ be the peaks of $\pi$. By Lemma 
\ref{lem:inserting_n_beforeandafter_peak},  
the sum of the contributions of inserting $n$ at $k_1-1, k_1, k_2-1, k_2 \dots,  k_r-1, k_r$ 
contributes $0$ in total.
		
The remaining gaps are $A = [n-1]_{0} - \{0,k_1-1, 
k_1, k_2-1, k_2, \dots,  k_r-1, k_r,n-2,n-1\}$. 
Again, it is easy to see that inserting $n$ at any gap from the set $A$ increases 
the number of peaks
and valleys by $1$. By Remark \ref{rem:parity _of_gaps}, the number of even gaps 
is $\floor{n/2}+1$ and the number of odd gaps is $\floor{n/2}$.  
Moreover, $0$ and $n-1$ are even gaps and $n-2$ is an odd gap. Hence, 
the number of even gaps available in the set $A$ is 
$ \floor{n/2}+1 - (r+2)$. Similarly, the number of odd gaps 
available in the set $A$ is $\floor{n/2}- (r+1)$.  Thus, the number
of odd and even gaps in $A$ are equal.  Hence, only the insertion at 
the final gap contributes to 
$\SR_{n,a,a}(p,q)$  Clearly, the contribution from 
$\pi' \in \SSS_{n-1,a,a} $ to  $\SR_{n,a,a}(p,q)$ is 
$(-1)^{\inv(\pi')}p^{\pkk(\pi')}q^{\vly(\pi')}$. 
		
\item To get a permutation of $\SSS_{n,a,a}$ from a permutation $\pi'$ of 
$\SSS_{n-1,a,d},$ there is only one choice: insert $n$ at the end of $\pi'$. 
Clearly, doing this increases the number of valleys by $1$. Further,
the number of inversions remains unchanged and therefore 
$\pi$ and $\pi'$ have same sign. 
Thus, the contribution from $\pi' \in \SSS_{n-1,a,d} $ to  
$\SR_{n,a,a}(p,q)$ is $q \times (-1)^{\inv(\pi')}p^{\pkk(\pi')}q^{\vly(\pi')}$. 
		
\item To get a permutation of $\SSS_{n,a,a}$ from a permutation $\pi'$ of 
$\SSS_{n-1,d,a},$ there is again only one choice:  insert $n$ at the 
first gap. 
Clearly, doing this increases the number of peaks by $1$.
Clearly, the number of inversions increases by $n-2$ which is 
odd and hence $\pi$ and $\pi'$ are of opposite sign. 
Thus, the contribution from $\pi' \in \SSS_{n-1,d,a} $ to  $\SR_{n,a,a}(p,q)$ 
is $(-p) \times (-1)^{\inv(\pi')}p^{\pkk(\pi')}q^{\vly(\pi')}$.
\end{enumerate}

	The proof of \eqref{eqn:naa_n_odd} is complete. 
%	The proof of the equations \eqref{eqn:nda_n_odd} and \eqref{eqn:ndd_n_odd}
%	will follow in a similar argument. 
Theorem \ref{thm:ad_equals_da_for_nminusone}  and \eqref{eqn:naa_n_odd}, 
give \eqref{eqn:naa_n_odd_simplified}. The proof of the Theorem is complete. 
\end{proof}

\subsection{When $n$ is even}
When $n$ is even, we have the following counterpart of 
Theorem \ref{thm:peakvalleyrecurrencesfornodd}.

\begin{theorem}
\label{thm:peakvalleyrecurrencesforneven}	
For even positive integers $n \geq 2$, we have: 
\begin{eqnarray}
\SR_{n,a,d}(p,q) &  = & -p\SR_{n-1,a,a}(p,q)  +p\SR_{n-1,d,d}(p,q)     \label{eqn:nad_n_is_even} \\
%	\SR_{n,d,a}(p,q) &  = & -q\SR_{n-1,a,a}(p,q)  +q\SR_{n-1,d,d}(p,q)  \label{eqn:nda_n_even} \\
	\SR_{n,a,a}(p,q) & = & (1+pq)\SR_{n-1,a,a}(p,q)+q\SR_{n-1,a,d}(p,q)  \nonumber\\
	& &+p\SR_{n-1,d,a}(p,q) \label{eqn:naa_n_even} 
	%& = & (1+pq)\SR_{n-1,a,a}(p,q)  \\
%	\SR_{n,d,d}(p,q) & = & -(1+pq)\SR_{n-1,d,d}(p,q)-p\SR_{n-1,d,a}(p,q) \nonumber \\
%	& &-q\SR_{n-1,a,d}(p,q)  \label{eqn:ndd_n_even} 
	%& = & (1+pq)\SR_{n-1,a,a}(p,q)  
	\end{eqnarray}
\end{theorem}

\begin{proof}
The moves we make in this proof are similar to those in the 
proof of Theorem \ref{thm:peakvalleyrecurrencesfornodd}.  Thus, we are
a little brief in our proofs of \eqref{eqn:nad_n_is_even} 
and \eqref{eqn:naa_n_even}. Clearly, we have 
\begin{eqnarray}
\SR_{n,a,d}(p,q) & = & \hspace{-3 mm} 
\sum_{\pi \in \SSS_{n,a,d}, \pi' \in \SSS_{n-1,a,d}} 
(-1)^{\inv(\pi)} p^{\pkk(\pi)}q^{\vly(\pi)} \nonumber \\ 
& &+ \sum_{\pi \in \SSS_{n,a,d}, \pi' \in \SSS_{n-1,a,a}} 
(-1)^{\inv(\pi)} p^{\pkk(\pi)}q^{\vly(\pi)} \nonumber \\ 
& & + \sum_{\pi \in \SSS_{n,a,d}, \pi' \in \SSS_{n-1,d,d}} 
(-1)^{\inv(\pi)} p^{\pkk(\pi)}q^{\vly(\pi)} \nonumber 
\end{eqnarray}
	We now record the contribution to $\SR_{n,a,d}(p,q)$ from the three terms separately.

\begin{enumerate}
\item Let, $\pi' \in \SSS_{n-1,a,d}$. To get $\pi \in \SSS_{n,a,d}$,
we cannot insert $n$ in the initial or final gap, 
as otherwise $\pi \not\in \SSS_{n,a,d}$.  Inserting $n$ in all the 
other gaps of  $\pi',$ results in $\pi \in \SSS_{n,a,d}$. 
As before let $k_1, k_2, \dots,  k_r$ be the peaks of $\pi$. 
By Lemma \ref{lem:inserting_n_beforeandafter_peak}, we can ignore the 
gaps in the set $k_1-1, k_1, k_2-1, k_2 \dots,  k_r-1, k_r$ as 
putting $n$ in these gaps contributes $0$ in total. 
		
The remaining gaps are $A = [n-1]_{0} - 
\{0,k_1-1, k_1, k_2-1, k_2, \dots,  k_r-1, k_r,n-1\}$. 
Inserting $n$ at any of those gaps increases both the number of 
peaks and valleys to increase by $1$.  By 
Remark \ref{rem:parity _of_gaps}, the number of even gaps is $n/2$ 
and the number of odd gaps is $n/2$.  Further, $0$ is an odd gap 
and $n-1$ is an even gap as $n-1$ is odd. Hence, the number of even gaps 
in $A$ is $ n/2 - r-1 $ and the number of odd gaps in $A$ is 
also $ n/2 - r-1 $. Thus, the contribution from 
$\pi' \in \SSS_{n-1,a,d} $ to  $\SR_{n,a,d}(p,q)$ is $0$.

\item Suppose $\pi' \in \SSS_{n-1,a,a}$ and we want to get $\pi \in 
\SSS_{n,a,d}$ by inserting $n$.
% in different gaps of $\pi'$ from a permutation of $\SSS_{n-1,a,d}.$ 
It is easy to see that this can be done in only one way: by inserting $n$ 
at the penultimate or $(n-2)$-nd gap.   This insertion increases the 
number of peaks by $1$. 
The number of inversions is also increased by $1$ and hence 
$\pi$ and $\pi'$ are of opposite sign. 
Thus, the contribution from $\pi' \in \SSS_{n-1,a,a} $ 
to  $\SR_{n,a,d}(p,q)$ is 
$(-p) \times (-1)^{\inv(\pi')}p^{\pkk(\pi')}q^{\vly(\pi')}$.
		
\item To get $\pi \in \SSS_{n,a,d}$ from $\pi' \in \SSS_{n-1,d,d},$ again
there is only one choice: insert $n$ at the first gap. This procedure
increases the number of peaks by $1$. 
As seen before, the number of inversions increases by $n-2$ which 
is even and hence $\pi$ and $\pi'$ are of same sign. 
Thus, the contribution from 
$\pi' \in \SSS_{n-1,d,d} $ to  $\SR_{n,a,d}(p,q)$ is 
$p \times (-1)^{\inv(\pi')}p^{\pkk(\pi')}q^{\vly(\pi')}$. 
\end{enumerate}
	
This completes the proof of \eqref{eqn:nad_n_is_even}. We now move on 
to prove \eqref{eqn:naa_n_even}. 
By Remark \ref{rem:adding_n_cannot_flip_both_end}, $\SR_{n,a,a}(p,q) $ 
is the sum of the following three polynomials.  
	
\begin{eqnarray}
\SR_{n,a,a}(p,q) & = & \hspace{-3 mm} 
\sum_{\pi \in \SSS_{n,a,a}, \pi' \in \SSS_{n-1,a,a}} 
(-1)^{\inv(\pi)} p^{\pkk(\pi)} q^{\vly(\pi)} \nonumber 
\\ & & + \sum_{\pi \in \SSS_{n,a,a}, \pi' \in \SSS_{n-1,a,d}} 
(-1)^{\inv(\pi)} p^{\pkk(\pi)} q^{\vly(\pi)} \nonumber \\ 
& & + \sum_{\pi \in \SSS_{n,a,a}, \pi' \in \SSS_{n-1,d,a}} 
(-1)^{\inv(\pi)} p^{\pkk(\pi)} q^{\vly(\pi)} \nonumber 
\end{eqnarray}
We record the contribution to $\SR_{n,a,a}(p,q)$ from these
three terms below. 
	
\begin{enumerate}
\item To get  $\pi \in \SSS_{n,a,a}$ from $\pi' \in \SSS_{n-1,a,a}$ 
we %nd we want to  get all inserting $n$ in different gaps of $\pi'$ 
%from a permutation of  $\SSS_{n-1,a,a}.$ 
cannot insert $n$ in the initial or penultimate gap, as 
otherwise $\pi \not \in \SSS_{n,a,a}$. Inserting $n$ at the final gap
of $\pi' \in \SSS_{n-1,a,a}$ is fine and this does not change 
the number of peaks and valleys. 
As before, let $k_1, k_2, \dots,  k_r$ be the peaks of $\pi$. 
By Lemma \ref{lem:inserting_n_beforeandafter_peak}, inserting $n$ in
the gaps $k_1-1, k_1, k_2-1, k_2 \dots,  k_r-1, k_r$ contributes 
$0$.
		
The remaining gaps are $A = [n-1]_{0} - \{0,k_1-1, k_1, k_2-1, k_2, 
\dots,  k_r-1, k_r,n-2,n-1\}$. 
We can insert $n$ at any of those gaps except the final gap 
and insertion increases the number of peaks and valleys by $1$. 
By Remark \ref{rem:parity _of_gaps}, both the number of even 
and odd gaps are $n/2$. As $0$ and $n-2$ are odd gaps and $n-1$ is 
an even gap, the number of even gaps in $A$ is $ n/2 - (r+1)$
while % Similarly, 
the number of odd gaps in $A$ is %such that inserting $n$ results 
% in the increase of peaks and valleys both by $1$ is 
$n/2- (r+2)$. Thus, the contribution from 
$\pi' \in \SSS_{n-1,a,a} $ to  $\SR_{n,a,a}(p,q)$ is 
$(1+pq) \times (-1)^{\inv(\pi')}p^{\pkk(\pi')}q^{\vly(\pi')}$. 
		
\item To get $\pi \in \SSS_{n,a,a}$ from $\pi' \in \SSS_{n-1,a,d},$ 
there is only one choice: insert $n$ at the end. 
As a result, the number of valleys increases by $1$.  Clearly,
the number of inversions remains unchanged and therefore 
$\pi$ and $\pi'$ have same sign.  Thus, 
the contribution from $\pi' \in \SSS_{n-1,a,d} $ to  
$\SR_{n,a,a}(p,q)$ is $q \times (-1)^{\inv(\pi')}p^{\pkk(\pi')}q^{\vly(\pi')}$. 
		
\item To get $\pi \in \SSS_{n,a,a}$ from $\pi' \in \SSS_{n-1,d,a},$ 
there is only one choice: insert $n$ at the first gap. 
Clearly, the number of peaks increases by 1.  Further, the 
number of inversions increase by $n-2$ which is even 
and hence $\pi$ and $\pi'$ are of same sign.
Thus, the contribution from $\pi' \in \SSS_{n-1,d,a} $ to  
$\SR_{n,a,a}(p,q)$ is $p \times (-1)^{\inv(\pi')}p^{\pkk(\pi')}q^{\vly(\pi')}$.
	\end{enumerate} 
	
The proof of \eqref{eqn:naa_n_even} is complete. 
\end{proof}

\section{Proof of Theorem \ref{thm:signed_peak_valley-enumeration}}

Towards proving Theorem \ref{thm:signed_peak_valley-enumeration}, 
we prove the following. 

\begin{theorem}
\label{thm:n0mod4biv}
For positive integers $n=4k,$ the following hold:
%\begin{enumerate}
%\item
\begin{eqnarray*}
%\label{itm:4kaa} 
\SR_{4k,a,a}(p,q) & = & \SR_{4k,d,d}(p,q)=(1+pq)(1-pq)^{2(k-1)}, \\
%\item 
%\label{itm:4kad} 
\SR_{4k,a,d}(p,q) & = & -2p(1-pq)^{2(k-1)}, \\
%\item 
%\label{itm:4kda}
\SR_{4k,d,a}(p,q) & = & -2q(1-pq)^{2(k-1)}, \\
%\item 
%\label{itm:4k}
\SR_{4k}(p,q) & = & 2(1-p)(1-q)(1-pq)^{2(k-1)}.
\end{eqnarray*}
%\end{enumerate}
\end{theorem}

\begin{proof}
We use induction on $k$. When $k=1$, the four statements can easily be
verified. We assume all the four statements to be true for $k=m$ and 
prove them for $k=m+1$. Let $n=4m+4$.  
We prove a seemingly 
digressive result first.  We claim that
\begin{eqnarray}
\SR_{4m+1,a,a}(p,q)=\SR_{4m,a,a}(p,q),  \label{eqn:nminus3_andnminus4aa} 
\\
	\SR_{4m+1,d,d}(p,q)=\SR_{4m,d,d}(p,q), \label{eqn:nminus3_andnminus4dd} 
\\
	\SR_{4m+1,a,d}(p,q)=\SR_{4m,a,d}(p,q), 	\label{eqn:nminus3_andnminus4ad} 
\\
	\SR_{4m+1,d,a}(p,q)=\SR_{4m,d,a}(p,q). \label{eqn:nminus3_andnminus4da} 
	\end{eqnarray}
We note that both  \eqref{eqn:nminus3_andnminus4aa}  and
\eqref{eqn:nminus3_andnminus4dd} follow from Theorem 
\ref{thm:peakvalleyrecurrencesfornodd}. By 
Theorem \ref{thm:peakvalleyrecurrencesfornodd}, we have
\begin{eqnarray}
\SR_{4m+1,a,d}(p,q) &  = & (-pq) \times \SR_{4m,a,d}(p,q) -p \times 
\SR_{4m,a,a}(p,q) \nonumber 
\\ & & -p \times \SR_{4m,d,d}(p,q) \nonumber  \\
& = & (-pq)(-2p)(1-pq)^{2(m-1)} -2p(1+pq)(1-pq)^{2(m-1)} \nonumber \\
& = & (-2p)(1-pq)^{2(m-1)} = \SR_{4m,a,d}(p,q) 
\end{eqnarray}
The second equality above follows by induction. Thus, 
\eqref{eqn:nminus3_andnminus4ad} holds. As the proof of 
\eqref{eqn:nminus3_andnminus4da} is identical, we omit it.
We now consider $n=4m+4$. We have
%	prove \ref{itm:4kaa}, \ref{itm:4kad}, \ref{itm:4kda} and \ref{itm:4k}.  
	\begin{eqnarray}
	\label{eqn:formulanaabiv}
	\SR_{4m+4,a,a}(p,q) & = & (1+pq)\SR_{4m+3,a,a}(p,q) \nonumber \\
	& = &  (1+pq)\SR_{4m+2,a,a}(p,q) \nonumber \\
	& = & (1+pq)\big[(1+pq)\SR_{4m+1,a,a}(p,q) \nonumber \\  & & + q\SR_{4m+1,a,d}(p,q)  +p\SR_{4m+1,d,a}(p,q)\big]  \nonumber \\
	& = &  (1+pq)\big[(1+pq)\SR_{4m,a,a}(p,q) \nonumber \\  & & + q\SR_{4m,a,d}(p,q) \nonumber    \nonumber +p\SR_{4m,d,a}(p,q)\big] \nonumber \\
	& = &  (1+pq)\big[(1+pq)(1+pq)(1-pq)^{2(m-1)} \nonumber \\ & &+ q(-2p)(1-pq)^{2(m-1)}  +p(-2q)(1-pq)^{2(m-1)}\big] \nonumber  \\& = &(1+pq)(1-pq)^{2m} .
	\end{eqnarray}
The first equality follows by Theorem 
\ref{thm:peakvalleyrecurrencesforneven} and Lemma \ref{lem:23mod4}. 
The second equality uses Theorem \ref{thm:peakvalleyrecurrencesfornodd}. In the third equality we again use  Theorem \ref{thm:peakvalleyrecurrencesforneven}. 
The fourth equality follows from %Equations 
\eqref{eqn:nminus3_andnminus4aa}, \eqref{eqn:nminus3_andnminus4ad} and 
\eqref{eqn:nminus3_andnminus4da}.  The fifth equality uses the 
induction hypothesis with $n=4m$. Similarly, we get
\begin{eqnarray}
\label{eqn:formulanadbiv}
\SR_{4m+4,a,d}(p,q) & = & (-p)\SR_{4m+3,a,a}(p,q) +p\SR_{4m+3,d,d}(p,q)\nonumber \\
& = &  -2p\SR_{4m+3,a,a}(p,q) \nonumber \\
& = &  -2p\SR_{4m+2,a,a}(p,q) \nonumber \\
& = & -2p(1-pq)^{2m} . 
\end{eqnarray}
Here too, the first equation follows from Theorem 
\ref{thm:peakvalleyrecurrencesforneven}.  The second equation uses 
Lemma \ref{lem:23mod4}. The third equation uses Theorem 
\ref{thm:peakvalleyrecurrencesfornodd}. The fourth equation uses the 
fact that $\SR_{4m+2,a,a}(p,q)=(1-pq)^{2m}$ 
which we already computed while proving \eqref{eqn:formulanaabiv}. 
In a similar way one can show that
\begin{eqnarray}
\label{eqn:formulandabiv}
\SR_{4m+4,d,a}(p,q) & = & -2q(1-pq)^{2m}.
\end{eqnarray}
Summing up \eqref{eqn:formulanaabiv}, \eqref{eqn:formulanadbiv} and
\eqref{eqn:formulandabiv} we get
\begin{eqnarray}
\SR_{4m+4}(p,q) & = & 2 \SR_{4m+4,a,a}(p,q) + \SR_{4m,a,d}(p,q) + \nonumber\\
& & + \SR_{4m,a,d}(p,q) \nonumber \\
& = &  2(1-p)(1-q)(1-pq)^{2m}.
\end{eqnarray}
The proof is complete. 
\end{proof}

From the proof of Theorem \ref{thm:n0mod4biv}, we get the 
following corollary. 
\begin{corollary}
\label{cor:signed_alt_run_in_4k_equald_signed_altrun_in_4kplus1}
For positive integers $k$, when $n=4k+1$, we have 
%\begin{enumerate}
\begin{eqnarray*}
%\item 
%\label{itm:4kplus1aa} 
\SR_{4k+1,a,a}(p,q) & = & \SR_{4k,d,d}(p,q)=(1+pq)(1-pq)^{2(k-1)}, \\
%\item 
%\label{itm:4kplus1ad} 
\SR_{4k+1,a,d}(p,q) & = & -2p(1-pq)^{2(k-1)}, \\
%\item 
%\label{itm:4kplus1da}
\SR_{4k+1,d,a}(p,q) & = & -2q(1-pq)^{2(k-1)}, \\
%\item 
%\label{itm:4kplus1}
\SR_{4k+1}(p,q) & = & 2(1-p)(1-q)(1-pq)^{2(k-1)}.
%\end{enumerate}	
\end{eqnarray*}

	%\begin{eqnarray}
	%	\SR_{n+1,a,a}(p,q)=\SR_{n,a,a}(p,q), \nonumber  \\
	%	\SR_{n+1,d,d}(p,q)=\SR_{n,d,d}(p,q), \nonumber \\
	%	\SR_{n+1,a,d}(p,q)=\SR_{n,a,d}(p,q), \nonumber  \\
	%	\SR_{n+1,d,a}(p,q)=\SR_{n,d,a}(p,q). \nonumber 
	%\end{eqnarray}
\end{corollary}

\noindent
{\bf Proof of Theorem 
\ref{thm:signed_peak_valley-enumeration} : }
The proof of Theorem follows directly from Theorem \ref{thm:n0mod4biv}, 
Corollary \ref{cor:signed_alt_run_in_4k_equald_signed_altrun_in_4kplus1}
and Lemma \ref{lem:23mod4}. 
{\hspace*{\fill}{\eod}}

From Theorem \ref{thm:signed_peak_valley-enumeration}, we immediately 
have the following corollary.

\begin{corollary}
\label{cor:signed_peak_equals_signed_valley_equals_0}
For positive integers $n$, we have
\begin{equation}
\sum_{\pi \in \SSS_n}(-1)^{\inv(\pi)}p^{\pkk(\pi)} = 0 = \sum_{\pi \in \SSS_n}(-1)^{\inv(\pi)}q^{\vly(\pi)}. 
\end{equation}
\end{corollary}

It is easy to see that peaks and valleys are equidistributed 
over $\SSS_n$. By Corollary 
\ref{cor:signed_peak_equals_signed_valley_equals_0}, we get that 
they are equidistributed over $\AAA_n$ and $\SSS_n - \AAA_n$.

\subsection{Multiplicity of $(1+t)$ as a factor of $R_n^{\pm}(t)$} 
Before we get the multiplicity of $(1+t)$ as a factor of 
the univariate polynomial $R_n^{\pm}(t)$, we get the signed 
univariate polynomial.  

\begin{corollary}
\label{thm:nmod4_uni}
For positive integers $k$, the following signed enumeration 
results hold.  When $n=4k$ and $n=4k+1$, we have:
%\begin{enumerate}
\begin{eqnarray*}
%\item 
\SR_{n,a,a}(t) & = &\SR_{n,d,d}(t)=t(1+t^2)(1-t^2)^{2k-2}, \\
%\item 
\SR_{n,a,d}(t) & = & \SR_{n,d,a}(t)=-2t^2(1-t^2)^{2k-2}, \\
%\item 
\SR_{n}(t) & = & 2t(1-t)^{2k}(1+t)^{2k-2}.
\end{eqnarray*}
%\end{enumerate}

When $n=4k+2$ and $n=4k+3$, we have:
%\begin{enumerate}
\begin{eqnarray*}
%\item 
\SR_{n,a,d}(t) & = & 0   \mbox{ and } \SR_{n,d,a}(t)=0,\\
%\item 
\SR_{n,a,a}(t) & = & -\SR_{n,d,d}(t) = t(1-t^2)^{2k}, \\
%\item 
\SR_{n}(t) & = & 0. 
\end{eqnarray*}
%\end{enumerate}
\end{corollary}

\begin{proof}
Follows from Theorem \ref{thm:signed_peak_valley-enumeration}
and the simple observation that $R_n(t) = t R_n(p,q)|_{p=q=t}.$
\comment{
Setting $p=q=t$ in Theorem \ref{thm:n0mod4biv}, 
we get that the results are true when $n=4k$. For $n=4k+1$, the 
proof follows from the proof of 
Equations \eqref{eqn:nminus3_andnminus4aa}, 
\eqref{eqn:nminus3_andnminus4dd}, \eqref{eqn:nminus3_andnminus4ad} 
and \eqref{eqn:nminus3_andnminus4da} which we have already 
proved while proving Theorem \ref{thm:n0mod4biv}. 
While proving Equation \eqref{eqn:formulanaabiv}, we proved that 
	$\SR_{4k+2,a,a}(p,q)=(1-pq)^{2k}$. Therefore,  
$\SR_{4k+2,a,a}(t)=t(1-t^2)^{2k}$.  
Now, the proof is complete using Lemma \ref{lem:23mod4}.}
\end{proof}

\vspace{3 mm}

We are now in a position to prove Theorem 
\ref{thm:Rnplusminust_is_divible_by_1plust_m}.

\noindent
{\bf Proof of Theorem 
\ref{thm:Rnplusminust_is_divible_by_1plust_m}: } 
We have 
\begin{eqnarray*}
R_n^{\pm}(t) & = & \frac{R_n(t) \pm \SR_n(t)}{2}, \\
&  = & \begin{cases}
\displaystyle \frac{R_n(t) \pm 2t(1-t)^{2k}(1+t)^{2k-2}}{2}& 
\text {when $n = 4k, 4k+1.$ }\\
\displaystyle \frac{R_n(t)}{2} & \text {when $n = 4k+2, 4k+3.$ }
\end{cases}
\end{eqnarray*}

In the second line above, we have used Corollary \ref{thm:nmod4_uni}.
Recall that $m = \lfloor (n-2)/2 \rfloor$.  Hence, when 
$n=4k,4k+1$, then $m = 2k-1$.  
%\blue{Can we write $m$ in terms of $k$?  That will help to see
%the exponent of $(1+t)$ more clearly}.  
By Theorem 
\ref{thm:divisibilty_of_alternatingrunpolynomial_by_highpower_oneplust}, 
the polynomial $R_n(t)$ is divisible by $(1 + t)^{m-1}$ and by 
Remark \ref{rem:evennessofcoefficientsofRnt}, the polynomial 
$R_n(t)$ is divisible by $2$. Thus, the polynomial $2(1 + t)^{m}$ 
divides the polynomial $R_n(t)$ but 
$2(1 + t)^{m-1}$ divides $\SR_n(t)$. Therefore, the polynomial 
$2(1 + t)^{m-1}$ divides $R_n^{\pm}(t)$ when $n \equiv 0,1$ (mod 4). 
Similarly, when $n \equiv 2,3 $ (mod 4), the polynomial 
$2(1 + t)^{m}$ divides the polynomials $R_n(t)$. 
This completes the proof. 
{\hspace*{\fill}{\eod}}

%\vspace{3 mm}
%Recall $m = \lfloor (n-2)/2 \rfloor$.  
Below, we show some examples 
when $n \equiv 0,1$ (mod 4) where $m-1$ is the largest
exponent of $(1+t)$ that divides $R_n^{\pm}(t)$.  

\begin{example} 
\label{eg_multiplicity-tight}
When $n=4$ and $n=5$, one can check that $(1+t)$ divides $R_4(t) $ and $R_5(t) $, 
but it does not 
divide the polynomials $R_4^+(t) , R_4^-(t), R_5^+(t) $ and $R_5^-(t)$.
        
\vspace{2 mm}    
\begin{tabular}{rl|rl}
$R_4(t)$ & $= 2t + 12t^2 + 10t^3 $  & $R_5(t)$ & = $ 2+28t^2+58t^3+32t^4$        \\
$R_4^+(t)$ & $= 2t + 4t^2 + 6t^3 $  & $R_5^+(t)$ & $=2+12t^2+30t^3+16t^4$        \\
$R_4^-(t)$ & $= 8t^2 + 4t^3 $  & $R_5^-(t)$ & $=16t^2+28t^3+16t^4$       
\end{tabular}
                
\vspace{2 mm}

Similarly when $n=8$, as the following data shows,
$(1+t)^3$ divides $R_8(t)$.
%but it does not 
%divide the polynomials $R_8^+(t)$ and $ R_8^-(t)$.  
However, only $(1+t)^2$ divides the 
polynomials $R_8^+(t)$ and $ R_8^-(t)$.
        
\vspace{-4 mm}
        
\begin{eqnarray*}
R_8(t) & = & 2t+252t^2+ 2766t^3+9576t^4+ 14622t^5+10332t^6+2770t^7 \\
& = & (1+t)^3\bigg( 2t+246t^2+2022t^3+2770t^4 \bigg), \\
R_8^+(t) & = & 2t+124t^2+ 1382t^3+4792t^4+ 7310t^5+5164t^6+1386t^7 \\ 
& = & (1+t)^2\bigg(2t+120t^2+1140t^3+2392t^4+1386t^5\bigg),\\
R_8^-(t) & = & 128t^2+ 1384t^3+4784t^4+ 7312t^5+5168t^6+1384t^7 \\
& = & (1+t)^2\bigg(128t^2+1128t^3+2400t^4+1384t^5\bigg).
\end{eqnarray*}
\end{example}

%\subsection{An explicit formula for $R_{n,k}^{\pm}$}
%These are by
%Stanley in \cite{stanleyalternatingsequence}, Canfield 
%and Wilf in \cite{canfieldwilfalternatingrun} and 
%Ma in \cite{maaltrun}. 
\begin{remark}
\label{rem:explicit-formula}
As mentioned in Section \ref{sec:intro}, 
three explicit formulae for $R_{n,\ell}$ are known.  
Using any of these, it is easy to get a formula 
for $R_{n,\ell}^{\pm}$.  Let $F_{n,\ell}$ denote any 
formula for $R_{n,\ell}$.  Combining 
$F_{n,\ell}$ with Corollary \ref{thm:nmod4_uni}, we get 
the following formula for $R_{n,\ell}^{\pm}$.  
%Since the proof is simple, we merely state our result and omit
%the proof.
%Theorem \ref{thm:explicitformulafornumberofaltrun}.
%\begin{theorem}
%\label{thm:explicitformulafornumberofaltrunplusminus}	
For positive integers $n,l$ with $n \geq 1$ and $1 \leq \ell \leq n-1$, 
let $F_{n,\ell}$ be a formula for the numbers $R_{n,\ell}$.
%we have 
%$$ R_{n,l}^{\pm}= \frac{1}{2}\Bigg[ \sum_{i=1}^l \frac{1}{2^{i-1}}(-1)^{l-i} z_{l-i} \sum _{r+2m \leq i, r \equiv i \hspace{1 mm} \textit{(mod 2)}} (-2)^m \binom{i-m}{(i+r)/2} \binom{n}{m} r^n + D_{n,l} \Bigg] , $$
%	where $z_0=2$ and $z_n=4$ for $n \geq 1 $ and $D_{n,l}$ is defined as follows:
Then, $R_{n,\ell}^{\pm} = \frac{1}{2}\big[ F_{n,\ell} + G_{n,\ell} \big]$ 
where
$$	G_{n,\ell}=\begin{cases}
\displaystyle 	(-4)(-1)^\frac{(\ell-2)}{2} \binom {2k-2}{\frac{\ell-2}{2}} 
	& \text {when $n=4k$ or $4k+1$ and $\ell$ is even}.\\
\displaystyle	2(-1)^\frac{\ell-1}{2} \binom {2k-2}{\frac{\ell-1}{2}} + 
2(-1)^\frac{\ell-3}{2} \binom {2k-2}{\frac{\ell-3}{2}} 
& \text {when $n=4k$ or $4k+1$ and $\ell$ is odd}.\\
		0 & \text {when $n=4k+2$ or $4k+3$}.
		\end{cases}  $$
\end{remark}
%\end{theorem}

\subsection{Moment-like identities}

In this subsection, we prove Theorem
\ref{thm:sum_of_kth_powers_of_odd_equals_sum_of_kth_powers_of_even}
which refines Lemma \ref{lem:moment-identity-chow-ma}.
we use the proof of the result of Chow and Ma
\cite[Corollary 2]{chowmaaltruntypeb}.  We have paraphrased their 
result, but from their proof, this change of form will be clear.

\begin{lemma}[Chow and Ma]
\label{lem:chow-ma-powers-alt}
If $(1+t)^m$ divides $f(t) = \sum_{i=0}^n f_i t^i$, then for 
positive integers $k \leq m-1$, we have 
$$1^k f_1 + 3^k f_3 + \cdots = 2^k f_2 + 4^k f_4 + \cdots.$$
\end{lemma}

\noindent
{\bf Proof of Theorem 
\ref{thm:sum_of_kth_powers_of_odd_equals_sum_of_kth_powers_of_even}: }
The proof now follows from Theorem 
\ref{thm:Rnplusminust_is_divible_by_1plust_m} and 
Lemma \ref{lem:chow-ma-powers-alt}.
{\hspace*{\fill}{\eod}}

\subsection{Alternating permutations in $\SSS_n^{\pm}$}%$\AAA_n$ and $\SSS_n-\AAA_n$}
  
Let $[t^k]f(t)$ denote the coefficient of $t^k$ in the polynomial
 $f(t)$.
Alternating permutations in $\SSS_n$ clearly have $n-1$ alternating 
runs and hence we get $E_n = [t^{n-1}] R_n(t)$.
We start with the following remark connecting alternating permutations
in $\AAA_n$ and two of the four sets that we have been working with.
  
\begin{remark}
\label{rem:Alt_perm_and_highest_coef_in_signedaltrun_poly} 
For even positive integers $n$, 
any $\pi \in \Alt_n$ must be in $\SSS_{n,d,d}$.  
In this case, we have 
$E_n^+-E_n^-=[t^{n-1}]\SR_{n,d,d}(t)$.
Similarly, for odd positive integers $n$, 
any $\pi \in \Alt_n$ must be in $\SSS_{n,d,a}$.  
In this case, we have 
$E_n^+-E_n^- =[t^{n-1}]\SR_{n,d,a}$. 
\end{remark}

We begin with the following simple lemma, whose proof follows
easily from Corollary \ref{thm:nmod4_uni}.  Since the proof
is easy, we omit it.

\begin{lemma}
\label{lem:rel_between_alt_and_altplusminus}
For positive integers $n \geq 2$, we have
\begin{eqnarray}
E_n^+ - E_n^-& = & \begin{cases}
1 & \text {if $n=4k$},\\
-1 & \text {if $n=4k+2$},\\
0 & \text {if $n=4k+1$ or $n=4k+3$.}
\end{cases}
\end{eqnarray} 
\end{lemma}
 
\comment{ 
\begin{proof}
The idea is to use  
Corollary \ref{thm:nmod4_uni} and
Remark \ref{rem:Alt_perm_and_highest_coef_in_signedaltrun_poly} together. 
For $n=4k$, we have
  	
  	\begin{equation}
  	E_n^+ - E_n^- =[t^{n-1}]\SR_{n,d,d}(t)= [t^{4k-1}]t(1+t^2)(1-t^2)^{2(k-1)} = 1. \nonumber
  	\end{equation} 
  	For $n=4k+1$, we have
  	\begin{equation}
  	E_n^+ - E_n^- =[t^{n-1}]\SR_{n,d,a}(t)= [t^{4k}]t(1+t^2)(1-t^2)^{2(k-1)} = 0. \nonumber
  	\end{equation}
  	For $n=4k+2$, we have 
  	\begin{equation}
  	E_n^+ - E_n^- =[t^{n-1}]\SR_{n,d,d}(t)= [t^{4k+1}](-t(1-t^2)^{2k}) = -1. \nonumber
  	\end{equation}
  	For $n=4k+3$, we have
  	\begin{equation}
  	E_n^+ - E_n^- =[t^{n-1}]\SR_{n,d,a}(t)= 0. \nonumber
  	\end{equation}
  	The proof is complete. 
  \end{proof}
 }
  
With this preparation, we can now prove  Theorem 
\ref{thm:egf_for_alt_perm_even_and_odd}.
  
\noindent
{\bf Proof of Theorem \ref{thm:egf_for_alt_perm_even_and_odd} :}
  	We have
\begin{eqnarray*}
\sum_{n=0}^{\infty} {E_n^{\pm}}\frac{x^n}{n!}	
& = & \frac{1}{2}\sum_{n=0}^{\infty} {E_n}\frac{x^n}{n!}  
\pm \frac{1}{2} \sum_{n=0}^{\infty}(E_n^+ - E_n^-) \frac{x^n}{n!} \nonumber \\
&  = & \frac{1}{2}(\sec x+ \tan x) \pm \frac{1}{2}\bigg( 1+x-\frac{x^2}{2!}+\frac{x^4}{4!}+\ldots \bigg) \\
  		& = &\frac{1}{2}(\sec x+\tan x \pm \cos x \pm x). 
  	\end{eqnarray*}
{\hspace*{\fill}{\eod}}

\section{Type B Coxeter Groups}

Recall that $\BB_n$ is the set of permutations of 
$[\pm n] = \{\pm 1, \pm 2,\ldots,\pm n\}$ satisfying $\pi(-i)= - \pi(i)$. 
For $\pi = \pi_1,\pi_2,\ldots,\pi_n \in \BB_n$, define 
$\Negs(\pi) = \{\pi_i : i > 0, \pi_i < 0 \}$ as
the set of elements which occur in $\pi$ with a negative sign.  
As defined in Petersen's book, \cite[Page 294]{petersen-eulerian-nos-book},
define
\begin{equation}
\label{eqn:typeb-inv-defn}
\inv_B(\pi)=  | \{ 1 \leq i < j \leq n : \pi_i > \pi_j \} |+  
| \{ 1 \leq i < j \leq n : -\pi_i > \pi_j \}| +|\Negs(\pi)|.  
\end{equation}
We referi to $\inv_B(\pi)$ alternatively as the {\it length 
of $\pi \in \BB_n$}.
Let $\BB^+_n \subseteq \BB_n$ 
denote the subset of even length elements of $\BB_n$ and 
let $\BB_n^- = \BB_n - \BB_n^+$. For $\pi = \pi_1,\pi_2,\ldots,\pi_n 
\in \BB_n $, let $ \pi_0=0$ and define its set of type B peaks 
and type B valleys to be 
$\Pkk_B(\pi)= \{ i \in [n-1]: \pi_{i-1} < \pi_{i} > \pi_{i+1} \}$ 
and $\Vly_B(\pi)=\{ i \in [n-1]: \pi_{i-1}> \pi_i < \pi_{i+1} \}$ 
respectively. Let $\pkk_B(\pi)= |\Pkk_B(\pi)|$ and 
$\vly_B(\pi)=|\Vly_B(\pi)|$ be the cardinality of these sets. 
For $\pi \in \BB_n$, we say that $\pi$ changes direction at index
$i$ if $i \in \Vly_B(\pi) \cup \Pkk_B(\pi)$.
We say that $\pi \in \BB_n$ has $k$ type B alternating 
runs, denoted as $\altr_B(\pi) = k$ if it changes 
direction a total of $k-1$ times.  For example, 
the permutation $5,1,4,\overline{3}, \overline{6},2 \in \BB_6$ has 
$\pkk_B(\pi) = 2, \vly_B(\pi)=2$ and hence has $\altr_B(\pi)=5.$
Define $R_n^B(t) = \sum_{\pi \in \BB_n} t^{\altr_B(\pi)} 
= \sum_{k=1}^n R_{n,k}^Bt^k.$
Similar to the $\SSS_n$ case, we are interested in 
enumerating the following signed analogue.
\begin{equation}
\label{eqn:type_b_pkvly_over_Bn}
\SBR_{n}(p,q)  =  \sum_{\pi \in \BB_n} (-1)^{\inv_B(\pi)} 
p^{\pkk_B(\pi) }q^{\vly_B(\pi)} .
\end{equation}

Recall that we had partitioned $\SSS_n$ into four sets. We similarly 
partition $\BB_n$  into the following two sets: 
$\BB_{n,-,a}= \{\pi \in \BB_n : \pi_{n-1} < \pi_n \}$ and 
$\BB_{n,-,d}= \{ \pi \in \BB_n: \pi_{n-1} > \pi_n \}$. 
Define the following polynomials: 
\begin{eqnarray}
\label{eqn:type_b_pkvly_over_Bnfinalascent}
\SBR_{n,-,a}(p,q) & = & \sum_{\pi \in \BB_{n,-,a} } (-1)^{\inv_B(\pi)} p^{\pkk_B(\pi)}q^{\vly_B(\pi)}, \\
 \label{eqn:type_b_pkvly_over_Bnfinaldescent}
 \SBR_{n,-,d}(p,q) & = & \sum_{\pi \in \BB_{n,-,d} } (-1)^{\inv_B(\pi)} p^{\pkk_B(\pi)}q^{\vly_B(\pi)}.
\end{eqnarray}
We start with the following type B counterparts of Lemma
\ref{lem:01mod4basic} and Lemma  \ref{lem:23mod4} and connect the
polynomials $\SBR_{n,-,a}(p,q)$ and $\SBR_{n,-,d}(p,q)$. 

\begin{lemma}
\label{lem:relation_between_ending_in_ascent_and_ending_in_descent}
For positive integers $n$, we have
\begin{eqnarray}
\label{eqn:odd_typeb-sbr-relation}
\SBR_{n,-,a}(p,q) & = & - \SBR_{n,-,d}(q,p)  \mbox{ when $n$ is odd.}  \\
\label{eqn:_typeb-sbr-relation}
\SBR_{n,-,a}(p,q) & = & \SBR_{n,-,d}(q,p) \mbox{ when $n$ is even.} 
\end{eqnarray}

\end{lemma} 

\begin{proof}
Let $\pi= \pi_1, \pi_2, \dots, \pi_n \in \BB_n$. 
Define the map $\sflip: \BB_{n,-,a} \mapsto \BB_{n,-,d}$ by
$$\sflip(\pi) = \overline{\pi_1}, \dots, \overline{\pi_2}, \dots,  
\overline{\pi_n}.$$ 
Clearly, $\sflip$ is a bijection from $\BB_{n,-,a}$ to
$\BB_{n,-,d}$. It is easy to see that flipping the sign of a 
single $\pi_i$ changes the parity of $\inv_B$ 
(for example, see \cite[Lemma 3]{siva-sgn_exc_hyp}). Therefore, the 
map $\comp_B$ preserves the parity of $\inv_B$ if and only if $n$ 
is even. Moreover, it is easy to 
see that $\pkk_B(\pi)= \vly_B(\comp_B(\pi))$ and 
$\vly_B(\pi)= \pkk_B(\comp_B(\pi))$. As the remaining portion
of the proof 
is similar to the proof of Lemma \ref{lem:01mod4basic} and 
Lemma  \ref{lem:23mod4}, we omit the details. 
The proof is complete. 
\end{proof}

%For $k \in [n]$ and $\pi \in \BB_n$, 
%let $\pos_r(\pi)=k$ if $\pi_{k} =r$ and 
%$\pos_{-r}(\pi)=k$ if $\pi_{k} =-r$. 
For $k \in [n]$ and $\pi \in \BB_n$, recall $\pi_k \in [\pm n]$.
We define $\pos_{\pi_k}(\pi) = k$.  
Any $r \in [n]$ appears exactly once as $|\pi_i|$ where $1 \leq i \leq n$.
Define $\pos_{\pm r}(\pi)=k$ if $|\pi_k| = |r|$. For example,
if $\pi = 5,1,4,\ol{3},\ol{6},2$, we have $\pos_{\pm 6}(\pi) = 5$ as
$|\pi_5| = |6|$.
%\blue{Why is $\pi_k > 0$?  If it is equal to $|r|$, then it
%should be positive, but why is this always so?} \red{Something
%here needs to be rewritten well.}

For  
$r \in [\pm n]$, let $\sign(r) \in \{ \pm 1\}$ denote the sign of 
$r$. For example, $\sign(-5)=-1$ but $\sign(5)=1.$
For $\pi=\pi_1,\pi_2, \dots ,\pi_n \in \BB_n$, let $\pi''$ be 
obtained from $\pi$ by deleting the letters $n$ and $n-1$. 
Suppose $\pi \in \BB_{n,-,a}$.  It is easy 
to see that both $\pi'' \in \BB_{n-2,-,a}$ and $\pi'' \in \BB_{n-2,-,d}$ 
are possible.

We partition $\BB_{n,-,a}$ into the following  
$8$ disjoint subsets and will consider the contribution of each 
set to $\SBR_{n,-,a}(p,q)$.  
The idea of partitioning of $\BB_{n,-,a}$  depends on two points.
First, for $\pi \in \BB_{n,-,a}$, whether $\pi'' \in \BB_{n,-,a}$ or 
$\pi'' \in \BB_{n,-,d}$.  The second point is whether the highest 
two letters in absolute value, that is,  
$\pm n$ and $\pm (n-1)$ are consecutive or not. 
As we will see in Lemmas
\ref{lem:contributiontobivpeakvalleyfrom_n_nminus1_separate}, 
\ref{lem:contributiontobivpeakvalleyfrom_n_nminus1_consectutive_but_not_last},
\ref{lemma:contribution_n_n-1_consec} and
\ref{lemma:contribution_n_n-1_consec_second}, seven of these 
terms %individually 
will contribute nothing and hence 
only one of these 8 terms will contribute to $\SBR_{n,-,a}(p,q)$.
%\red{More intuitive names other than $\BB^1$.. $\BB^8$ would
%be helpful below.}

\begin{enumerate}
%	\item $\SSS_n^2= \{ \pi \in \SSS_n : \pi^{-1}(n) \neq n, \pi^{-1}(n-1) \neq n, |\pi^{-1}(n) - \pi^{-1}(n-1)| =1 \}$, 
\item $\BB_{n,-,a}^1= \{ \pi \in \BB_{n,-,a} :\pi'' \in \BB_{n-2,-,a} 
\mbox{ and }  |\pos_{\pm n}(\pi) - \pos_{\pm (n-1)}(\pi)| > 1  \}.$
\item $\BB_{n,-,a}^2= \{ \pi \in \BB_{n,-,a} :\pi'' \in \BB_{n-2,-,d} 
\mbox{ and }  |\pos_{\pm n}(\pi) - \pos_{\pm (n-1)}(\pi)| > 1  \}.$
\item $\BB_{n,-,a}^3= \{ \pi \in \BB_{n,-,a} :\pi'' \in \BB_{n-2,-,a}
\mbox{ with } |\pos_{\pm n}(\pi) - \pos_{\pm (n-1)}(\pi)| =1 
\mbox{ and }  \pi_n \notin \{ \pm (n-1), \pm  n  \} \}.$
\item $\BB_{n,-,a}^4= \{ \pi \in \BB_{n,-,a} :\pi'' \in \BB_{n-2,-,d}
\mbox{ with } 
|\pos_{\pm n}(\pi) - \pos_{\pm (n-1)}(\pi)| =1 
\mbox{ and } 
\pi_n \notin \{  \pm (n-1), \pm  n  \} \}.$
\item  $\BB_{n,-,a}^5= \{ \pi \in \BB_{n,-,a} :\pi'' \in \BB_{n-2,-,a}
\mbox{ with } 
|\pos_{\pm n}(\pi) - \pos_{\pm (n-1)}(\pi)|=1 
\mbox{ and } 
 \pi_n \in \{  \pm (n-1), \pm  n  \}, \sign(\pi_{n-1}) \neq \sign(\pi_{n})   \}.$
\item  $\BB_{n,-,a}^6= \{ \pi \in \BB_{n,-,a} :\pi'' \in \BB_{n-2,-,d}
\mbox{ with } 
|\pos_{\pm n}(\pi) - \pos_{\pm (n-1)}(\pi)|=1 
\mbox{ and } 
\pi_n \in \{  \pm (n-1), \pm  n  \}, \sign(\pi_{n-1}) \neq \sign(\pi_{n})   \}.$
\item   $\BB_{n,-,a}^7= \{ \pi \in \BB_{n,-,a} :\pi'' \in \BB_{n-2,-,d} 
\mbox{ with } 
|\pos_{\pm n}(\pi) - \pos_{\pm (n-1)}(\pi)|=1 
\mbox{ and } 
\pi_n \in \{  \pm (n-1), \pm  n  \}, \sign(\pi_{n-1}) = \sign(\pi_{n})   \}.$
\item $\BB_{n,-,a}^8= \{ \pi \in \BB_{n,-,a} :\pi'' \in \BB_{n-2,-,a}
\mbox{ with } 
|\pos_{\pm n}(\pi) - \pos_{\pm (n-1)}(\pi)|=1 
\mbox{ and } 
\pi_n \in \{  \pm (n-1), \pm  n  \}, \sign(\pi_{n-1}) = \sign(\pi_{n})   \}.$	
		%|\pi^{-1}(n) - \pi^{-1}(n-1)| \geq 2\}$,
		%	\item $\SSS_n^4= \{ \pi \in \SSS_n : \pi^{-1}(n)=n \hspace{3 mm} \mbox{or}
		%	\hspace{3 mm} \pi^{-1}(n-1)=n, |\pi^{-1}(n) - \pi^{-1}(n-1)| =1 \}$. 
	%|\pi^{-1}(n) - \pi^{-1}(n-1)| \geq 2\}$,
	%	\item $\SSS_n^4= \{ \pi \in \SSS_n : \pi^{-1}(n)=n \hspace{3 mm} \mbox{or}
	%	\hspace{3 mm} \pi^{-1}(n-1)=n, |\pi^{-1}(n) - \pi^{-1}(n-1)| =1 \}$. 
\end{enumerate}

We start showing that all except one of these sets contribute nothing to 
$\SBR_{n,-,a}(p,q)$.

\begin{lemma}
\label{lem:contributiontobivpeakvalleyfrom_n_nminus1_separate}
For positive integers $n \geq 3$, the contribution of both sets 
$\BB_{n,-,a}^1$ and $\BB_{n,-,a}^2$  to $\SBR_{n,-,a}(p,q)$ is $0$. That is, 
$$\sum_{\pi \in \BB_{n,-,a}^1} (-1)^{\inv_B(\pi)}
p^{\pkk_B(\pi)}q^{\vly_B(\pi)}=0, \hspace{10 mm} 
\sum_{\pi \in \BB_{n,-,a}^2} (-1)^{\inv_B(\pi)}
p^{\pkk_B(\pi)}q^{\vly_B(\pi)}=0.$$
\end{lemma}

\begin{proof}
We first consider the contribution of $\BB_{n,-,a}^1$.
Let $\pi = \pi_1, \dots ,\pi_{i-1},\pi_i= x, \pi_{i+1}, \dots, 
\pi_j=y, \dots, \pi_n \in \BB_{n,-,a}^1$, where $\{ |x|, |y| \} =\{n,n-1\}$.
%or $x=\overline{n}$ and 
Define $g: \BB_{n,-,a}^1 \mapsto \BB_{n,-,a}^1$ by 
$$g(\pi)= \begin{cases}
\pi_1, \dots, ,\pi_{i-1},\pi_i= y, \pi_{i+1}, \dots, \pi_j=x, \dots, \pi_n & \text {if $\sign(x)= \sign(y)$}, \\
\pi_1, \dots, ,\pi_{i-1},\pi_i= \overline{y}, \pi_{i+1}, \dots, \pi_j=\overline{x}, \dots, \pi_n & \text {if $\sign(x) \neq  \sign(y)$}.
\end{cases}$$
The map $g$ clearly preserves the sets $\Pkk_B$ and $\Vly_B$. Thus, we
have $\pkk_B(\pi) = \pkk_B(g(\pi))$ and $\vly_B(\pi) = \vly_B(g(\pi))$. 
When $\sign(x)=\sign(y)$, $g$ flips the parity of $\inv_B$ as the pair 
$(i,j)$ flips being an inversion. When $\sign(x) \neq \sign(y)$, we have
\begin{eqnarray} 
\label{eqn:reversinginginversion_B}
\inv_B(g(\pi)) & =  &  \inv_{B}( \pi_1, \dots, ,\pi_{i-1},\pi_i=\overline{y},\pi_i+1,\dots, \pi_j= \overline{x}, \dots, \pi_n)    \nonumber \\
& \equiv &  \inv_{B}( \pi_1, \dots, ,\pi_{i-1},\pi_i= y, \pi_{i+1}=x, \pi_{i+2}, \dots, \pi_n) \> (\hspace{-4 mm} \mod 2) \nonumber \\
& \equiv & \inv_B(\pi) -1 \>\>\> (\hspace{-4 mm} \mod 2).
\end{eqnarray}
The second step uses the fact that  
flipping the sign of a single $\pi_i$ changes the parity of the 
number of type B inversions (see \cite[Lemma 3]{siva-sgn_exc_hyp}).
%The third line uses the fact that swapping two letters of $\pi$ 
%changes the parity of type B inversions. 
Therefore, the map $g$ always reverses the parity of $\inv_B$. 
Thus the set $\BB_{n,-,a}^1$ contributes $0$  
to $\SBR_{n,-,a}(p,q)$. In an identical manner, one can show 
that the contribution of $\BB_{n,-,a}^2$
to $\SBR_{n,-,a}(p,q)$ is $0$. This completes the proof. 
\end{proof}

\begin{lemma}
\label{lem:contributiontobivpeakvalleyfrom_n_nminus1_consectutive_but_not_last}
For positive integers $n \geq 3$, the contribution of the two 
sets $\BB_{n,-,a}^3$ and $\BB_{n,-,a}^4$ to $\SBR_{n,-,a}(p,q)$ 
is $0$. That is, 
$$\sum_{\pi \in \BB_{n,-,a}^3} 
(-1)^{\inv_B(\pi)}p^{\pkk_B(\pi)}q^{\vly_B(\pi)}=0, \hspace{10 mm} 
\sum_{\pi \in \BB_{n,-,a}^4} 
(-1)^{\inv_B(\pi)}p^{\pkk_B(\pi)}q^{\vly_B(\pi)}=0$$
\end{lemma}

\begin{proof}
We first consider the contribution of $\BB_{n,-,a}^3$ to $\SBR_{n,-,a}(p,q)$.
Let $\pi \in \BB_{n,-,a}^3$. 
%Without loss of generality we assume 
%$\pos_{\pm n}(\pi) - \pos_{\pm { n-1}}(\pi)=-1$.   
Thus, $\pi$ has the following form: 
$$\pi= \pi_1, \dots, ,\pi_{i-1},\pi_i= x, \pi_{i+1}=y, \pi_{i+2}, 
\dots, \pi_n,$$ where $\{|x|,|y| \} = \{n,n-1\}$. %and %or $x= \overline{n}$ and 
%$|y|=n-1$. % or $y=\overline{n-1}$. 
	Define  $f: \BB_{n,-,a}^3 \mapsto \BB_{n,-,a}^3$ by 
	$$f(\pi)= \begin{cases}
	\pi_1, \dots, ,\pi_{i-1},\pi_i= y, \pi_{i+1}=x, \dots, \pi_n & \text {if $\sign(x)= \sign(y)$}, \\
	\pi_1, \dots, ,\pi_{i-1},\pi_i= \overline{y}, \pi_{i+1}= \overline{x}, \dots, \pi_n & \text {if $\sign(x) \neq  \sign(y)$}.
	\end{cases}$$
By an argument similar to that given in the proof of Lemma 
\ref{lem:contributiontobivpeakvalleyfrom_n_nminus1_separate}, one 
can show that $f$ flips the parity of $\inv_B$. 
Though $f$ does not preserve the sets $\Pkk_B$ and $\Vly_B$, 
it preserves the numbers $\pkk_B$ and $\vly_B$.  We show this
below.
%\blue{Some more details needed above.} 
%We give a small argument for the fact that the map $f$ 
%preserves $\pkk_B$. 
%{\bf Case 1: when $\sign(x)=\sign(y)$:} 
When $\sign(x)=\sign(y)$, changes among 
$\pkk_B(\pi)$ and $\pkk_B(f(\pi))$, clearly, can only occur in 
the $4$-element strings $\pi_{i-1}, \pi_{i}=x, \pi_{i+1}=y, 
\pi_{i+2}$ and   $\pi_{i-1}, \pi_{i}=y, \pi_{i+1}=x, \pi_{i+2}.$   
Both the $4$-element strings have $1$ peak when $x>0$ and 
both of them have no peaks when $x<0$. Similarly, when 
$\sign(x) \neq \sign(y)$, changes in 
$\pkk_B(\pi)$ and $\pkk_B(f(\pi))$ can only occur in the 
$4$-element strings $\pi_{i-1}, \pi_{i}=x, \pi_{i+1}=y, 
\pi_{i+2}$ and   $\pi_{i-1}, \pi_{i}=\ol{y}, 
\pi_{i+1}=\ol{x}, \pi_{i+2}.$ In this case, both  have 
$1$ peak. Thus, $\pkk_B(\pi)=\pkk_B(f(\pi)).$  The arguments
showing that $f$ preserves $\vly_B$ are similar and hence
omitted.  %and so we omit that.
Thus, $$\sum_{\pi \in \BB_{n,-,a}^3} 
(-1)^{\inv_B(\pi)}p^{\pkk_B(\pi)}q^{\vly_B(\pi)}=0.$$ 
The proof that the set $\BB_{n,-,a}^4$ also contributes $0$ is 
again similar and hence omitted. 
\end{proof}

\begin{lemma}
\label{lemma:contribution_n_n-1_consec}
For positive integers $n \geq 3$, the contribution of the two 
sets $\BB_{n,-,a}^5$ and $\BB_{n,-,a}^6$  to  $\SBR_n(p,q)$ is $0$.
That is,   
$$\sum_{\pi \in \BB_{n,-,a}^5} 
(-1)^{\inv_B(\pi)}p^{\pkk(\pi)}q^{\vly(\pi)}=0, \hspace{10 mm} 
\sum_{\pi \in \BB_{n,-,a}^6} 
(-1)^{\inv_B(\pi)}p^{\pkk(\pi)}q^{\vly(\pi)}=0.$$
\end{lemma}

\begin{proof}
Define $h_1: \BB_{n,-,a}^5 \mapsto \BB_{n,-,a}^5$ by
$$h_1(\pi_1, \pi_2, \dots, \pi_{n-1}=x, \pi_{n}=y) = \pi_1, \pi_2, \dots, \pi_{n-1}=\overline{y}, \pi_{n}=\overline{x}, $$
where $\{|x|,|y| \} = \{n,n-1 \}.$  The map $h_1$ clearly preserves the statistics
$\pkk_B$ and $\vly_B$ but changes the parity of $\inv_B$. Thus,
$\BB_{n,-,a}^5$ contributes $0$ to $\SBR_{n,-,a}(p,q)$. 
Similarly, one can show that  the set $\BB_{n,-,a}^6$ also contributes $0$ to $\SBR_{n,-,a}(p,q)$. The proof is complete. 
\end{proof}

\begin{lemma}
\label{lemma:contribution_n_n-1_consec_second}
For positive integers $n \geq 3$, the contribution of $\BB_{n,-,a}^7$ to  $\SBR_n(p,q)$ is $0$, that is,   
$$\sum_{\pi \in \BB_{n,-,a}^7} (-1)^{\inv_B(\pi)}p^{\pkk_B(\pi)}q^{\vly_B(\pi)}=0.$$ 
\end{lemma}

\begin{proof}
%Define $h_2: \BB_{n,-,a}^7 \mapsto \BB_{n,-,a}^7$ by
%$$h_2(\pi_1, \pi_2, \dots, \pi_{n-1}=n-1, \pi_{n}=n) = \pi_1, \pi_2, \dots, \pi_{n-1}=\overline{n}, \pi_{n}=\overline{n-1}. $$
It is easy to see that the same map $h_1$ defined in the proof 
of Lemma \ref{lemma:contribution_n_n-1_consec} works in this case, completing
the proof.
%preserves $\pkk_B$ and $\vly_B$ but 
%changes the parity of $\inv_B$. The proof is complete.	
\end{proof}

Thus, the contribution of $\BB_{n,-,a}^k$ for $1 \leq k \leq 7 $ 
to $\SBR_{n,-,a}(p,q)$ is $0$. Hence 

\begin{equation}
\label{eqn:only_BnA8_contributes_nonzero}
\SBR_{n,-,a}(p,q)= \sum_{\pi \in \BB_{n,-,a}^8} 
(-1)^{\inv_B(\pi)}p^{\pkk_B(\pi)}q^{\vly_B(\pi)}
\end{equation}

Inside each of $\BB_{n,-,a/d}^8$, we define yet another subset 
which has a nice inductive structure for signed enumeration and outside 
which signed enumeration gives zero.  We define
$T_{n,-,a} \subseteq \BB_{n,-,a}^8$ and $T_{n,-,d} \subseteq \BB_{n,-,d}^8$ 
such that outside $T_{n,-,a}$ and $T_{n,-,d}$, the 
signed enumeration equals $0$.

We inductively define the sets $T_{n,-,a}$  and $T_{n,-,d}$
as follows. When $n=1$, let $T_{1,-,a}= \{1\}$ and 
$T_{1,-,d}= \overline{1}$.  When $n=2$, let  $T_{2,-,a} = 
\{12,  \overline{2}\overline{1} \}$ where we have removed the commas
within each permutation for ease of reading. 
Let $T_{2,-,d} = \{21, \overline{1}\overline{2} \}$. 
%Consider $\pi=\pi_1, \pi_2, \dots, \pi_n \in T_{n-2,-,a}$. 
Consider $\pi \in T_{n-2,-,a}$. 
Using $\pi$, we form two signed permutations 
$\tau_1$ and $\tau_2 \in T_{n,-,a}$ by appending two 
letters as follows: %For each $r=1,2,3,4$ 
%and $i \in [n-2]$, define $\psi_r(i) = \pi(i)$. 
	%\tau_1 = \pi_1,\pi_2,\dots, \pi_{n-2},n-1,n, \hspace{15 mm}
	%\tau_2 = \pi_1,\pi_2,\dots, \pi_{n-2},\overline{n},\overline{n-1}. 
\begin{equation*}
	\tau_1 = \pi,n-1,n, \hspace{15 mm}
	\tau_2 = \pi,\overline{n},\overline{n-1}. 
\end{equation*}
In a similar manner, 
define the sets $T_{n,-,d}$ for $n \geq 2$ as follows:
%Let $\pi=\pi_1, \pi_2, \dots, \pi_n \in T_{n-2,-,d}$. 
Let $\pi \in T_{n-2,-,d}$. 
Using $\pi$, we form the two signed permutations 
$\tau_3$ and $\tau_4 \in T_{n,-,d}$ as follows: 
\begin{equation*}
	%\tau_3 = \pi_1,\pi_2,\dots, \pi_{n-2},n,n-1, \hspace{15 mm}
	%\tau_4 = \pi_1,\pi_2,\dots, \pi_{n-2},\overline{n-1},\overline{n}. 
	\tau_3 = \pi,n,n-1, \hspace{15 mm}
	\tau_4 = \pi,\overline{n-1},\overline{n}. 
\end{equation*}
In our next lemma, we show that elements of 
$\BB_{n,-,a}^8 - T_{n,-,a}$ contribute zero to   
$\SBR_{n,-,a}(p,q).$

\begin{lemma}
\label{lem:B_minus_T-equals_zero}
For positive integers $n$, the following is true.
\begin{equation}
\label{eqn:yet_another_cancel-a}
\sum _{\pi \in \BB^8_{n,-,a} - T_{n,-,a}} 
(-1)^{\inv_B(\pi)}p^{\pkk_B(\pi)}q^{\vly_B(\pi)}  =  0.
\end{equation}
\end{lemma}

\begin{proof}
We induct on $n$. It is easy to see when $n=1,2$ that 
$T_{1,-,a}= \BB^{8}_{1,-,a}$ and $T_{2,-,a}= \BB^{8}_{2,-,a}$. 
Thus, when $n\leq 2$, we trivially have 
$\sum _{\pi \in \BB^8_{n,-,a} - T_{n,-,a}} 
(-1)^{\inv_B}(\pi)p^{\pkk_B(\pi)}q^{\vly_B(\pi)}=0.$ 
By induction, assume this to be true for any positive integer 
$n=m$.  We will prove it for $n=m+2$. We have
\begin{eqnarray*}
& & \sum _{\pi \in \BB^8_{m+2,-,a} } 
(-1)^{\inv_B(\pi)} p^{\pkk_B(\pi)}q^{\vly_B(\pi)} \\ 
& = & \sum _{\pi \in \BB^8_{m+2,-,a}, \pi '' \in  \BB_{m,-,a} }  
(-1)^{\inv_B(\pi)}p^{\pkk_B(\pi)}q^{\vly_B(\pi)}, \\ 
& =  &   \sum _{\pi \in \BB^8_{m+2,-,a}, \pi '' \in  \BB^8_{m,-,a}  } 
(-1)^{\inv_B(\pi)}p^{\pkk_B(\pi)}q^{\vly_B(\pi)} \\  
&  & + \sum _{\pi \in \BB^8_{m+2,-,a}, \pi '' \in (\BB_{m,-,a}- \BB^8_{m,-,a})  } 
(-1)^{\inv_B(\pi)}p^{\pkk_B(\pi)}q^{\vly_B(\pi)}, \\ 
& = &  (1-pq)\big[ \sum _{\pi \in \BB^8_{m,-,a}  } 
(-1)^{\inv_B(\pi)}p^{\pkk_B(\pi)}q^{\vly_B(\pi)} +  
\sum _{ \pi  \in (\BB_{m,-,a}- \BB^8_{m,-,a})  }   
(-1)^{\inv_B(\pi)}p^{\pkk_B(\pi)}q^{\vly_B(\pi)} \big] ,    \\ 
& = &   (1-pq) \sum _{\pi \in T_{m,-,a} } 
(-1)^{\inv_B(\pi)}p^{\pkk_B(\pi)}q^{\vly_B(\pi)} + 0 =  
\sum _{\pi \in T_{m+2,-,a} } 
(-1)^{\inv_B(\pi)}p^{\pkk_B(\pi)}q^{\vly_B(\pi)}     
\end{eqnarray*}

The second line follows from the definition of $\BB^8_{m+2,-,a}$. 
The penultimate line follows from observing that a permutation 
$\pi \in \BB^8_{m+2,-,a}$ can arise from $\pi'' \in \BB^8_{m,-,a}$ in only 
two ways: either by putting the string $n-1,n$ at last, which does not 
change the number of peaks and valleys, or by putting 
$\overline{n}, \overline{n-1}$  which increases both the number 
of peaks and valleys by $1$ and flips the sign as well. The last 
line follows by the induction hypothesis, the recursive definition 
of $T_{m+2,-,a}$ and \eqref{eqn:only_BnA8_contributes_nonzero}. 
The proof is complete. 
\end{proof}

\begin{remark}
\label{rem:asc_to_des_map}
Using the map $\sflip$, all results involving $B_{n,-,a}^k$ for 
$1\leq k \leq 8$ go through with identically defined sets 
$B_{n,-,d}^k$ with our  definition of the set $T_{n,-,d}$.  
Thus, we can get versions of Lemmas
\ref{lem:contributiontobivpeakvalleyfrom_n_nminus1_separate},
\ref{lem:contributiontobivpeakvalleyfrom_n_nminus1_consectutive_but_not_last},
\ref{lemma:contribution_n_n-1_consec},
\ref{lemma:contribution_n_n-1_consec_second} and 
\ref{lem:B_minus_T-equals_zero}.  
\end{remark}
We are now in a position to give our type B counterpart of Theorem
\ref{thm:signed_peak_valley-enumeration}.

\begin{theorem}
	\label{thm:main_theorem_signed_peakvalley_enumerate_over_type_b}
	For positive integers $n \geq 2$, the following recurrence relations hold: 
	\begin{eqnarray}
	\label{eqn:recurrence_by_jump_2_peakvalley_B_a}
	\SBR_{n,-,a}(p,q) & = & (1-pq)\SBR_{n-2,-,a}(p,q), \\
	 \label{eqn:recurrence_by_jump_2_peakvalley_B_d}
	\SBR_{n,-,d}(p,q) & = &(1-pq)\SBR_{n-2,-,d}(p,q)
	\end{eqnarray}
	Therefore, we have 
\begin{eqnarray*}
	%$$
 \SBR_{n,-,a}(p,q) & = & \begin{cases}
	(1-q)(1-pq)^{k-1}%\frac{n-2}{2} 
			& \text {when $n = 2k$ is even},\\
	(1-pq)^k %\frac{n-1}{2} 
			& \text {when $n = 2k+1$ is odd}.
	\end{cases}  \\
		\SBR_{n,-,d}(p,q) & = & \begin{cases}
		(1-p)(1-pq)^{k-1} %\frac{n-2}{2} 
			& \text {when $n = 2k$ is even},\\
		-(1-pq)^k %\frac{n-1}{2} 
			& \text {when $n = 2k+1$ is odd}.
		\end{cases}  \\
		\SBR_{n}(p,q) & =& \begin{cases}
		(2-p-q)(1-pq)^{k-1} %\frac{n-2}{2} 
			& \text {when $n = 2k$ is even},\\
		0 & \text {when $n$ is odd}.
		\end{cases}  
\end{eqnarray*}
%$$
\end{theorem}

\begin{proof}
We first consider \eqref{eqn:recurrence_by_jump_2_peakvalley_B_a}.  
By Lemma  \ref{lem:B_minus_T-equals_zero}, we have
\begin{eqnarray}
\SBR_{n,-,a}(p,q) & = & %\sum_{\pi \in \BB_{n,-,a}^8} (-1)^{\inv_B(\pi)}p^{\pkk_B(\pi)}q^{\vly_B(\pi)} \nonumber \\ %& = & 
\sum_{\pi \in T_{n,-,a}} (-1)^{\inv_B(\pi)}p^{\pkk_B(\pi)}q^{\vly_B(\pi)} \nonumber \\
& = & (1-pq)\sum_{\pi \in T_{n-2,-,a}} (-1)^{\inv_B(\pi)}p^{\pkk_B(\pi)}q^{\vly_B(\pi)} \nonumber 
 \end{eqnarray}
Recall that for
$\pi= \pi_1, \pi_2, \dots, \pi_{n-2} \in T_{n-2,-,a}$, we have
$\tau_1= \pi_1, \pi_2, \dots, \pi_{n-2},n-1,n$  and
$\tau_2= \pi_1, \pi_2, \dots, \pi_{n-2}, \overline{n}, \overline{n-1}$.
The second line follows from the following easy to prove facts. 
We have 
$\pkk(\tau_1) = \pkk(\pi)$ and $\vly(\tau_1) = \vly(\pi)$.
Further, we have 
$\pkk(\tau_2) = \pkk(\pi) + 1$ and $\vly(\tau_2) = \vly(\pi) + 1$.
%and no extra peaks and valleys while the permutation 
%has 1 extra peak and 1 extra valley. 
This completes the proof of \eqref{eqn:recurrence_by_jump_2_peakvalley_B_a}.
The proof of \eqref{eqn:recurrence_by_jump_2_peakvalley_B_d} 
follows from \eqref{eqn:recurrence_by_jump_2_peakvalley_B_a} 
by using the $\sflip$ map.

The following base cases are also easy to see.  When $n=1,$ 
we have $\SBR_{n,-,a}(p,q)=1$ and $\SBR_{n,-,d}(p,q)=-1.$
Similarly, when $n=2$, 
we have $\SBR_{2,-,a}(p,q)= (1-q)$ 
and $\SBR_{2,-,d}(p,q) = (1-p)$.
Using the recurrences \eqref{eqn:recurrence_by_jump_2_peakvalley_B_a}, 
\eqref{eqn:recurrence_by_jump_2_peakvalley_B_d} and Lemma 
\ref{lem:relation_between_ending_in_ascent_and_ending_in_descent}, 
the proof is complete. 
\end{proof}

\subsection{Multiplicity of $(1+t)$ as a factor of $R_n^{B,\pm}(t)$}
 
Setting $p=q=t$ in Theorem 
\ref{thm:main_theorem_signed_peakvalley_enumerate_over_type_b} and 
multiplying the result by $t$, we get the following.

\begin{corollary}
\label{cor:univariate_signed_altrun_over_type_b}
For positive integers $n$, we have
$$\SBR_{n}(t)=\begin{cases}
2t(1-t)(1-t^2)^{k-1} %\frac{n-2}{2} 
	& \text {when $n = 2k$ is even},\\
0 & \text {when $n$ is odd}.
\end{cases}  $$
\end{corollary}

Zhao in \cite{zhaothesis} 
%Zhao had actually 
%proved a refinement of Theorem 
%\ref{thm:wilf-type-b-counterpart} by
considering the following polynomial
$$R_n^{B,>}(t)= \sum _{\pi \in \BB_n: \pi_1>0} 
t^{\altr_B(\pi)}=\sum_{i=1}^n R_{n,i}^{B,>} t^i.$$ 
In \cite[Theorem 4.3.2]{zhaothesis}, Zhao proved the 
following.

\begin{theorem}[Zhao]
\label{thm:machowaltrunB}
For positive integers $n$, the polynomials $R_n^{B,>}(t)$ are
divisible by $(1+t)^m$ where $m = \nmhalf$. 
%{\lfloor \frac{n-1}{2}\rfloor} $. 
\end{theorem}
  
Further, define the polynomials
$R_n^{B,<}(t)  = \sum_{\pi \in \BB_n, \pi_1<0} 
t^{\altr_B(\pi)}=\sum_{i=1}^n R_{n,i}^{B,<} t^i$.
From Theorem \ref{thm:machowaltrunB}, the following type B
counterpart of Theorem
\ref{thm:divisibilty_of_alternatingrunpolynomial_by_highpower_oneplust}
is easy to infer.  The statement of Theorem 
\ref{thm:wilf-type-b-counterpart} is implicit in the work
of Zhao, but is not made explicitly.
We thus present a short proof for
completeness.
 %, \hspace{14 mm} 

%a sharper result (see Theorem 
%\ref{thm:machowaltrunB}) from which the following 
%type B counterpart of Theorem 
%\ref{thm:divisibilty_of_alternatingrunpolynomial_by_highpower_oneplust}
%can be inferred (see Remark \ref{rem:firstnegativeequalsfirstpositive}).

\begin{theorem}[Zhao]
\label{thm:wilf-type-b-counterpart}
For positive integers $n$, the polynomial $R_n^B(t)$ is divisible by 
$(1 + t)^m$, where $m = \nmhalf$. %\floor{(n -1)/2}$. 
\end{theorem}
\begin{proof}
%\begin{remark}
%\label{rem:firstnegativeequalsfirstpositive}
Using the map $\sflip$, it is easy to see that 
$R_n^{B,<}(t)= R_n^{B,>}(t).$  Combining with 
Theorem \ref{thm:machowaltrunB}, we get that 
both $R_n^{B,<}(t)$ and  $R_n^{B,>}(t)$ 
are divisible by $(1+t)^{\nmhalf}$.
Therefore, $R_n^{B}(t)$  is divisible by  $2(1+t)^{\nmhalf}$.
%\end{remark}  
\end{proof}

Towards refining Theorem \ref{thm:wilf-type-b-counterpart},
we define the polynomials
\begin{equation*}
R_n^{B,+}(t) = \sum_{\pi \in \BB_n^+} 
t^{\altr_B(\pi)}=\sum_{i=1}^n R_{n,i}^{B,+} t^i \hspace{5 mm} \mbox{ and } \hspace{5 mm}
R_n^{B,-}(t)= \sum_{\pi \in \BB_n^-} t^{\altr_B(\pi)}
=\sum_{i=1}^n R_{n,i}^{B,-} t^i. 
\end{equation*}
  
From Corollary \ref{cor:univariate_signed_altrun_over_type_b},
we get the following refinement of Theorem 
\ref{thm:wilf-type-b-counterpart}.
  
\begin{theorem}
\label{thm:divisibiltyforplusminuspolynomialstypeB}
For positive integers $n$, both polynomials $R_n^{B, \pm }(t)$  
are divisible by $(1+t)^m$ where $m = \nmhalf$.
\end{theorem} 
  
\begin{proof} 
When $n$ is odd, by 
Corollary \ref{cor:univariate_signed_altrun_over_type_b}, we have
$R_n^{B,\pm}(t) = \frac{1}{2} R_n^B(t)$.  
When $n=2k$ is even,
by Corollary \ref{cor:univariate_signed_altrun_over_type_b}, we have
\begin{eqnarray*}
R_n^{B, \pm}(t) & = & \frac{R_n^B(t) \pm \SBR_n(t)}{2}
=   \frac{R_n^B(t) \pm 2t(1-t)^{2k}(1+t)^{2k-2}}{2}.
\end{eqnarray*}
The proof in both cases is complete by using 
Theorem \ref{thm:wilf-type-b-counterpart}. 
\end{proof}

\subsection{Moment-type identities in $\BB_n^{\pm}$}
\label{subsec:moment-type-typeb}

Chow and Ma in \cite[Corollary 5]{chowmaaltruntypeb} 
used Lemma \ref{lem:chow-ma-powers-alt} and Theorem   
\ref{thm:machowaltrunB}, to infer the following moment
type equality.

\begin{theorem}[Chow and Ma]
\label{thm:sum_of_kth_powers_of_odd_equals_sum_of_kth_powers_signed_type_b}
For positive integers $n$ and positive integers $k$ with $n \geq 2k+3$, 
\begin{eqnarray*}
\label{eqn:sum_of_kth_powers_of_odd_equals_sum_of_kth_powers_of_even_sign_type_b}
1^kR_{n,1}^{B,>}+3^kR_{n,3}^{B,>}+5^kR_{n,5}^{B,>}+\cdots 
& = & 
2^kR_{n,2}^{B,>}+4^kR_{n,4}^{B,>}+6^kR_{n,6}^{B,>}+\cdots, 
\end{eqnarray*}
\end{theorem}

From Theorem \ref{thm:wilf-type-b-counterpart} and 
Lemma \ref{lem:chow-ma-powers-alt}, the following slightly
weaker corollary follows.

\begin{corollary}
\label{thm:moment_identities_type_b}
For positive integers $n$ and positive integers $k$ with $n \geq 2k+3$, 
\begin{eqnarray*}
\label{eqn:sum_of_kth_powers_of_odd_equals_sum_of_kth_powers_of_even_type_b}
1^kR_{n,1}^{B}+3^kR_{n,3}^{B}+5^kR_{n,5}^{B}+\cdots 
& = & 
2^kR_{n,2}^{B}+4^kR_{n,4}^{B}+6^kR_{n,6}^{B}+\cdots, 
\end{eqnarray*}
\end{corollary}

From Lemma 
\ref{lem:chow-ma-powers-alt} and Theorem   
\ref{thm:divisibiltyforplusminuspolynomialstypeB}, 
the following refinement of Corollary
\ref{thm:moment_identities_type_b} is straightforward. 

\begin{theorem}
\label{thm:sum_of_kth_powers_of_odd_equals_sum_of_kth_powers_signed_type_b_evenodd}
For positive integers $n$ and positive integers $k$ with $n \geq 2k+3$, 
\begin{equation*}
\label{eqn:sum_of_kth_powers_of_odd_equals_sum_of_kth_powers_of_even_sign_type_b_plus}
1^kR_{n,1}^{B, \pm }+3^kR_{n,3}^{B, \pm}+5^kR_{n,5}^{B, \pm}+\cdots  
=  
2^kR_{n,2}^{B, \pm}+4^kR_{n,4}^{B, \pm}+6^kR_{n,6}^{B, \pm}+\cdots 
\end{equation*}
\end{theorem}

A signed refinement of Theorem 
\ref{thm:sum_of_kth_powers_of_odd_equals_sum_of_kth_powers_signed_type_b}
is also true.  However, the proof fits in the theme of another of our
work.  We thus refer the reader to 
\cite[Theorem 13]{dey-siva-grp-action-bona-wilf} 
for this.

\subsection{Alternating permutations in $\BB_n^{\pm}$}
Our results here on the type B counterparts of Theorem 
\ref{thm:egf_for_alt_perm_even_and_odd}
are very closely related to the type D counterparts.
Hence, we have moved them to Subsection \ref{subsec:alt_perms-bd}.

\section{Type D Coxeter Groups}

Recall that $\DD_n \subseteq \BB_n$ is the subset of type B 
permutations that have an even number of negative signs. 
Let $\pi = \pi_1, \pi_2,\ldots, \pi_n \in \DD_n$.
The following combinatorial definition of type D inversions
is well known (see, for example, Petersen's book 
\cite[Page 302]{petersen-eulerian-nos-book}):
$\inv_D(\pi) = \inv_A(\pi) + | \{1 \leq i < j 
\leq n:  -\pi_i > \pi_j \}|$.
Here $\inv_A(\pi)$ is computed with respect to the usual order
on $\mathbb{Z}$.  
Since these definitions are combinatorial, we can apply
them to elements $\pi \in \BB_n$ or to $\pi \in \DD_n$.  Hence, we also
have $\inv_D(\pi)$ when $\pi \in \BB_n$ and specially if $\pi \in \BB_n-\DD_n$.
From the above definition of $\inv_B(\pi)$ and $\inv_D(\pi)$ we
get the following simple corollary which we will need later. 
Comparing the above definition of $\inv_D$ with 
\eqref{eqn:typeb-inv-defn}, we get the following elementary 
corollary.

\begin{corollary}
\label{cor:invb_invd_ddn_bbn}
Let $\pi \in \BB_n$.  Then, $\inv_B(\pi) = \inv_D(\pi) + |\Negs(\pi)|$.
Hence, if $\pi \in \DD_n$, then 
$\inv_B(\pi) \equiv \inv_D(\pi)$ (mod 2). 
\end{corollary}

Let $\DD_n^+ = 
\{\pi \in \DD_n: \inv_D(\pi) \> \mbox{ is even} \}$  
and let $\DD_n^- = \DD_n - \DD_n^+$. 
Let $(\BB_n - \DD_n)^+ = 
\{\pi \in \BB_n - \DD_n: \inv_D(\pi) \> \mbox{ is even} \}$  
and let $(\BB_n-\DD_n)^- = (\BB_n-\DD_n)- (\BB_n-\DD_n)^+$. It is 
easy to see that  $(\BB_n - \DD_n)^+ = \BB_n^- -  \DD_n^-$ 
and  $(\BB_n - \DD_n)^- = \BB_n^+ -  \DD_n^+$.  We will use
this in the proof of Theorem \ref{thm:divisibiltyfortyped}.
For $\pi \in \DD_n$, we have the same definition of peak, valley and 
alternating run as in $\BB_n$. That is, 
$\pkk_D(\pi)= \pkk_B(\pi)$, $\vly_D(\pi)=\vly_B(\pi)$ 
and $\altr_D(\pi)=\altr_B(\pi)$. 
We define the following polynomials.
\begin{eqnarray*}
R_n^{D}(t)  &  = & \sum_{\pi \in \DD_n} t^{\altr_D(\pi)}=\sum_{i=1}^n R_{n,i}^{D} t^i , \hspace{5 mm}
R_n^{B-D}(t)  =  \sum_{\pi \in \BB_n- \DD_n} t^{\altr_D(\pi)}=\sum_{i=1}^n R_{n,i}^{B-D} t^i, \\
R_n^{D,+}(t) & = &\sum_{\pi \in \DD_n^+} t^{\altr_D(\pi)}=\sum_{i=1}^n R_{n,i}^{D,+} t^i , \hspace{5 mm} R_n^{D,-}(t)= \sum_{\pi \in \DD_n^-} t^{\altr_D(\pi)}=\sum_{i=1}^n R_{n,i}^{D,-} t^i, \\
R_n^{B-D,+}(t) & = &\sum_{\pi \in (\BB_n - \DD_n)^+} t^{\altr_D(\pi)} =\sum_{i=1}^n R_{n,i}^{B-D,+} t^i , \\ R_n^{B-D,-}(t) & = &\sum_{\pi \in (\BB_n - \DD_n)^-} t^{\altr_D(\pi)}=\sum_{i=1}^n R_{n,i}^{B-D,-} t^i. 
\end{eqnarray*}

Since we wish to enumerate a signed version, we also define
\begin{equation}
\label{eqn:type_d_pkvly_over_Dn}
\SDR_{n}(p,q)   =  \sum_{\pi \in \DD_n} (-1)^{\inv_D(\pi)} 
p^{\pkk_D(\pi)}q^{\vly_D(\pi)} .
\end{equation}

As in the type B case, we partition $\DD_n$ into the following 
two sets $\DD_{n,-,a}= \{\pi \in \BB_n : \pi_{n-1} < \pi_n \}$ 
and $\DD_{n,-,d}= \{ \pi \in \BB_n: \pi_{n-1} > \pi_n \}$. 
We next define approprite polynomials: 

\begin{eqnarray}
\label{eqn:type_d_pkvly_over_Bnfinalascent}
\SDR_{n,-,a}(p,q) & = & \sum_{\pi \in \DD_{n,-,a} } 
(-1)^{\inv_D(\pi)} p^{\pkk_D(\pi)}q^{\vly_D(\pi)}, \\
\label{eqn:type_d_pkvly_over_Bnfinaldescent}
\SDR_{n,-,d}(p,q) & = & \sum_{\pi \in \DD_{n,-,d} } 
(-1)^{\inv_D(\pi)} p^{\pkk_D(\pi)}q^{\vly_D(\pi)}.
\end{eqnarray}

%\begin{lemma}
%\label{lem:relation_between_ending_in_ascent_and_ending_in_descent_type_d}
%For positive integers $n$, we have
%$$\SDR_{n,-,a}(p,q)=  \SDR_{n,-,d}(q,p).$$
%\end{lemma} 
%
%\begin{proof}
%	Let $\pi= \pi_1, \pi_2, \dots, \pi_n \in \DD_n$. 
%	Define the map $Bar: \BB_{n,-,a} \mapsto \BB_{n,-,d}$ by
%	$$Bar(\pi) = \overline{\pi_1}, \dots, \overline{\pi_2}, \dots,  \overline{\pi_n}.$$
%	We observe that $\pkk_B(\pi)= \vly_B(Bar(\pi))$ and $\vly_B(\pi)= \pkk_B(Bar(\pi))$. The map $Bar$ preserves its parity of $\inv_D$ for any $n$. The proof is complete. 
%\end{proof}

Suppose $\pi \in \DD_{n,-,a}$.  As in the type B case, let
$\pi''$ be obtained from $\pi$ by deleting the largest two
elements in absolute value.
Then, it is easy to see that both $\pi'' \in \DD_{n-2}$ and 
$\pi'' \in \BB_{n-2} - \DD_{n-2}$ are possible.  From 
$\pi'' \in \DD_{n-2}$,  
to get $\pi \in \DD_n$, we have to insert either 
$n$ and $n-1$ or $\overline{n}$ and $\overline{n-1}$.
Similarly, to get $\pi \in \DD_n$ from $\pi'' \in 
(\BB_{n-2}-\DD_{n-2})$, we have to either insert
$n$ and $\overline{n-1}$ or $\overline{n}$ and $n-1$.
A similar statement is true  when we wish to get 
$\pi \in \BB_n - \DD_n$  from either 
$\pi'' \in \BB_{n-2} - \DD_{n-2}$ or from $\pi'' \in \DD_{n-2}$.

Similar to the type B case, we partition $\DD_{n,-,a}$ into 
the following $9$ disjoint subsets and compute the
contribution of each set to $\SDR_{n,-,a}(p,q)$. 
The reason for partitioning $\DD_{n,-,a}$ into $9$ sets 
instead of $8$ sets as done with $\BB_{n,-,a}$ is the following: 
for $\pi \in \DD_{n,-,a}$, $\pi''$ may be outside $\DD_{n-2}$. 
This results in the extra one case. 
As done in the type B case, we will see that only one
set contributes to $\SDR_{n,-,a}(p,q)$.  We will see that the 
proof for seven of the cases carry over from the type B case. 
See Remark \ref{rem:Type_b_maps_going_on_for_type_d}.

\begin{enumerate}
\item $\DD_{n,-,a}^1= \{ \pi \in \DD_{n,-,a} :\pi'' \in \DD_{n-2,-,a} 
\mbox{ with }
 |\pos_{\pm n}(\pi) - \pos_{\pm (n-1)}(\pi)| > 1  \}$.

\item $\DD_{n,-,a}^2= \{ \pi \in \DD_{n,-,a} :\pi'' \in \DD_{n-2,-,d} 
\mbox{ with }
  |\pos_{\pm n}(\pi) - \pos_{\pm (n-1)}(\pi)| > 1  \}$.

\item $\DD_{n,-,a}^3= \{ \pi \in \DD_{n,-,a} :\pi'' \in \DD_{n-2,-,a}
\mbox{ with }
 |\pos_{\pm n}(\pi) - \pos_{\pm (n-1)}(\pi)| =1 
\mbox{ and }
 \pi_n \notin \{ \pm (n-1), \pm  n  \} \} $.

\item $\DD_{n,-,a}^4= \{ \pi \in \DD_{n,-,a} :\pi'' \in \DD_{n-2,-,d}
\mbox{ with }
 |\pos_{\pm n}(\pi) - \pos_{\pm (n-1)}(\pi)| =1 
\mbox{ and }
\pi_n \notin \{  \pm (n-1), \pm  n  \} \} $.

\item  $\DD_{n,-,a}^5= \{ \pi \in \DD_{n,-,a} :\pi'' \in \DD_{n-2,-,a}
\mbox{ with }
|\pos_{\pm n}(\pi) - \pos_{\pm (n-1)}(\pi)|=1, 
 \pi_n \in \{  \pm (n-1), \pm  n  \} 
\mbox{ and }
\sign(\pi_{n-1}) \neq \sign(\pi_{n})   \}$.

\item  $\DD_{n,-,a}^6= \{ \pi \in \DD_{n,-,a} :\pi'' \in \DD_{n-2,-,d} 
\mbox{ with }
|\pos_{\pm n}(\pi) - \pos_{\pm (n-1)}(\pi)|=1, 
 \pi_n \in \{  \pm (n-1), \pm  n  \}
\mbox{ and }
\sign(\pi_{n-1}) \neq \sign(\pi_{n})   \}$.

\item   $\DD_{n,-,a}^7= \{ \pi \in \DD_{n,-,a} :\pi'' \in \DD_{n-2,-,d} 
\mbox{ with }
|\pos_{\pm n}(\pi) - \pos_{\pm (n-1)}(\pi)|=1 , \pi_n \in \{  \pm (n-1), \pm  n  \}
\mbox{ and }
\sign(\pi_{n-1}) = \sign(\pi_{n})   \}$.

\item $\DD_{n,-,a}^8= \{ \pi \in \DD_{n,-,a} :\pi'' \in \DD_{n-2,-,a} 
\mbox{ with }
|\pos_{\pm n}(\pi) - \pos_{\pm (n-1)}(\pi)|=1 , \pi_n \in \{  \pm (n-1), \pm  n  \} 
\mbox{ and }
\sign(\pi_{n-1}) = \sign(\pi_{n})   \}$.

\item $\DD_{n,-,a}^9= \{ \pi \in \DD_{n,-,a} :\pi'' \in \BB_{n-2} - \DD_{n-2} \}$. 
\end{enumerate} 

The following remark settles the contribution of seven of these nine sets. 

\begin{remark}
\label{rem:Type_b_maps_going_on_for_type_d}
Consider the 
maps $g$, $f$ and $h_1$ constructed in the proofs of 
Lemmas  \ref{lem:contributiontobivpeakvalleyfrom_n_nminus1_separate}, 
\ref{lem:contributiontobivpeakvalleyfrom_n_nminus1_consectutive_but_not_last},
\ref{lemma:contribution_n_n-1_consec} and 
\ref{lemma:contribution_n_n-1_consec_second}.  It is clear that
all of them preserve the number 
of negative elements $|\Negs(\pi)|$ and 
we have proved that they flip the parity of $\inv_B(\pi)$.
Thus, by Corollary \ref{cor:invb_invd_ddn_bbn}, when viewed as 
maps from $\DD_{n,-,a} \mapsto \DD_{n,-,a}$, all the four maps 
flip the parity of $\inv_D$. Thus, 
the contribution of $\DD_{n,-,a}^k$  
for $1 \leq k \leq 7 $ to $\SDR_{n,-,a}(p,q)$ is $0$. 
\end{remark}

We next  prove that the contribution of $\DD_{n,-,a}^9$ to 
$\SDR_{n,-,a}(p,q)$ is also $0$. 

\begin{lemma}
\label{lem:contributiontotypedfromnminustwobminusd}
For positive integers $n \geq 3$, the contribution of $\DD_{n,-,a}^9$ 
to  $\SDR_n(p,q)$ is $0$, that is,   
$$\sum_{\pi \in \DD_{n,-,a}^9} 
(-1)^{\inv_D(\pi)}p^{\pkk_D(\pi)}q^{\vly_D(\pi)}=0.$$ 
\end{lemma}

\begin{proof}
Let $\pi \in \DD_{n,-,a}^9$  be such that $\pi '' \in \BB_{n-2}- \DD_{n-2}$. 
Thus, $\pi$ either contains $n$ and $\overline{n-1}$ or $\overline{n}$ and 
$n-1$.  Without loss of generality, let 
$$\pi = \pi_1, \dots, ,\pi_{i-1},\pi_i= n, \pi_{i+1}, \dots, \pi_j=\overline{n-1}, \dots, \pi_n .$$ Define 
$h_3: \DD_n^6 \mapsto \DD_n^6$ by 
$h_3(\pi)= \pi_1, \dots, ,\pi_{i-1},\pi_i= n-1, \pi_{i+1}, \dots, \pi_j=\overline{n}, \dots, \pi_n$.
It is easy to check that the map $h_3$ preserves the statistics $\pkk_D$ 
and $\vly_D$ but changes the parity of $\inv_D$. The proof is complete.
\end{proof}

From Remark \ref{rem:Type_b_maps_going_on_for_type_d} and Lemma 
\ref{lem:contributiontotypedfromnminustwobminusd} we have 
$$\SDR_{n,-,a}(p,q)= \sum_{\pi \in \DD_{n,-,a}^8} 
(-1)^{\inv_D(\pi)} p^{\pkk_D(\pi)} q^{\vly_D(\pi)}.$$

%Note that the set $\DD_{n,-,a}^8$ actually coincides with the 
%set $\BB_{n,-,a}^8$ and so we have the following result.
%\red{Why are the two sets the same?  I am not sure about this, 
%especially when $n$ is odd.}
%\ogreen{Hiranya will write and send.}

\magenta{In our next lemma, we 
claim that bivariate signed enumeration 
of the number of peaks and valleys
over $\BB^8_{n,-,a} - \DD^8_{n,-,a}$, 
is zero. This is a counterpart of Lemma
\ref{lem:B_minus_T-equals_zero} to the type D case.
}

\begin{lemma}
\label{lem:biv_enum_over_D8minusT_equals_0}
For positive integers $n$, we have 
$$\sum _{\pi \in \DD^8_{n,-,a} - T_{n,-,a}} 
(-1)^{\inv_D(\pi)}p^{\pkk_D(\pi)}q^{\vly_D(\pi)}=0.$$
\end{lemma}
\begin{proof}
The proof of this Lemma is identical to the proof of Lemma 
\ref{lem:B_minus_T-equals_zero}.  It is simple to note 
that  $T_{1,-,a} \subseteq \DD^8_{1,-,a}$ and 
$T_{2,-,a} \subseteq \DD^8_{2,-,a}$ defined just 
before Lemma 
\ref{lem:B_minus_T-equals_zero} are subsets of 
$\DD_{n,-,1}^8$.  That is, $T_{n,-,a} \subseteq \DD^8_{n,-,a}$.  
Now, the proof of Lemma 
\ref{lem:B_minus_T-equals_zero} carries over identically.
\end{proof}

\begin{remark}
\label{rem:asc_to_des_d-type}
Using the map $\sflip$, all results involving $D_{n,-,a}^k$ for 
$1\leq k \leq 9$ go through with identically defined sets 
$D_{n,-,d}^k$.  We have also defined the set $T_{n,-,d}$.  
Thus, we can get versions of Lemmas
\ref{lem:contributiontobivpeakvalleyfrom_n_nminus1_separate},
\ref{lem:contributiontobivpeakvalleyfrom_n_nminus1_consectutive_but_not_last},
\ref{lemma:contribution_n_n-1_consec},
\ref{lemma:contribution_n_n-1_consec_second} and 
\ref{lem:B_minus_T-equals_zero}.  
\end{remark}
%\magenta{Here also we can comment the above Remark.}
With this preparation, we can now prove the following:

\begin{theorem}
\label{thm:main_theorem_signed_peakvalley_enumerate_over_type_d}
For positive integers $n$, we have 
\begin{eqnarray*}
\SDR_{n,-,a}(p,q) & = &\begin{cases}
(1-q)(1-pq)^{k-1} % \frac{n-2}{2} 
	& \text {when $n = 2k$ is even},\\
(1-pq)^k %\frac{n-1}{2} 
	& \text {when $n = 2k+1$ is odd}.
\end{cases}  \\
\SDR_{n,-,d}(p,q) & = &\begin{cases}
(1-p)(1-pq)^{k-1} %\frac{n-2}{2} 
	& \text {when $n = 2k$ is even},\\
0 & \text {when $n$ is odd}.
\end{cases}  
\end{eqnarray*}
\end{theorem}

\begin{proof}
Consider $\SDR_{n,-,a}(p,q)$.  Both results when $n$ is odd and even 
follow from Theorem 
\ref{thm:main_theorem_signed_peakvalley_enumerate_over_type_b}, Corollary
\ref{cor:invb_invd_ddn_bbn} and Lemma
\ref{lem:biv_enum_over_D8minusT_equals_0}.  

For  $\SDR_{n,-,d}(p,q)$, we make 
use of the map $\sflip$. 
First consider the case when $n = 2k$ is even. It is easy to see
that $\sflip: \DD_{n,-,d} \mapsto\DD_{n,-,a}$ is a bijection.
Thus,  we have 
$$\SDR_{n,-,d}(p,q)= \SDR_{n,-,a}(q,p)=(1-p)(1-pq)^{k-1}.$$
When $n$ is odd, it is again clear that 
$\sflip: \DD_{n,-,d} \mapsto \BB_{n,-,a} - \DD_{n,-,a}$
is a bijection. In this case, we have 
\begin{eqnarray*} 
\SDR_{n,-,d}(p,q) & = & \sum_{\pi \in \DD_n, \pi_{n-1}> \pi_n}
(-1)^{\inv_D(\pi)}p^{\pkk_D(\pi)}q^{\vly_D(\pi)}, \\
& = & \sum_{\pi \in \BB_n-\DD_n, \pi_{n-1}< \pi_n}
(-1)^{\inv_D(\pi)} p^{\pkk_D(\pi)}q^{\vly_D(\pi)}, \\
& = & \SBR_{n,-,a}(p,q)-\SDR_{n,-,a}(p,q) 
=  0.
\end{eqnarray*}
The proof is complete.
\end{proof}

\subsection{Multiplicity of $(1+t)$ as a factor of $R_n^{D,\pm}(t)$}

Setting $p=q=t$ in Theorem 
\ref{thm:main_theorem_signed_peakvalley_enumerate_over_type_d}, 
and multiplying by $t$, we get the following corollary.

\begin{corollary}
	\label{cor:univariate_signed_altrun_over_type_d}
	For positive integers $n$, we have
	$$\SDR_{n}(t)=\begin{cases}
	2t(1-t)(1-t^2)^{k-1} %\frac{n-2}{2} 
			& \text {when $n = 2k$ is even},\\
	t(1-t^2)^k %\frac{n-1}{2} 
			& \text {when $n = 2k+1$ is odd}.
	\end{cases}  $$
\end{corollary}
Moving to the multiplicity of $(1+t)$ as a divisor of 
$R_n^D(t)$, Gao and Sun in \cite{Gaosunaltruntyped} %
proved a refined version involving $R_n^{D,>}(t)$. 
\comment{
\begin{theorem}[Gao and Sun]
\label{thm:wilf-type-d-counterpart}
For positive integers $n$, the polynomial $R_n^D(t)$ is divisible by 
$(1 + t)^m$, where $m = \floor{(n -1)/2}$. 
\end{theorem}
We begin by first defining the polynomials
\begin{equation*}
R_n^{D,+}(t) = \sum_{\pi \in \DD_n^+} 
t^{\altr_B(\pi)}=\sum_{i=1}^n R_{n,i}^{D,+} t^i \hspace{5 mm} 
\mbox{ and } \hspace{5 mm}
R_n^{D,-}(t)= \sum_{\pi \in \DD_n^-} t^{\altr_B(\pi)}
=\sum_{i=1}^n R_{n,i}^{D,-} t^i. 
\end{equation*}
}
Gao and Sun 
considered the following polynomials: 
\begin{eqnarray*}
R_n^{D,>}(t) & = & \sum_{\pi \in \DD_n, \pi_1>0} 
t^{\altr_D(\pi)} = \sum_{i=1}^n R_{n,i}^{D,>} t^i ,\\
R_n^{B-D,>}(t) & = &  \sum_{\pi \in \BB_n-\DD_n, \pi_1>0} 
t^{\altr_D(\pi)}=\sum_{i=1}^n R_{n,i}^{B-D,>} t^i.
\end{eqnarray*}
In \cite[Corollary 2.3]{Gaosunaltruntyped} 
they proved the following result regarding the polynomial
$R_n^{D,>}(t)$. 
\begin{theorem}[Gao and Sun]
\label{thm:Gao_sun_type_dfirstcoordinatepositive}
For positive integers $n$, we have
\begin{equation}
\label{eqn:difference_dfirstpositive_minus_bminusd_first_positive}
R_n^{D,>}(t) - R_n^{B-D,>}(t)  =\begin{cases}
t(1-t)(1-t^2)^{k-1} %\frac{n-2}{2} 
		& \text {when $n = 2k$ is even},\\
t(1-t^2)^k %\frac{n-1}{2} 
		& \text {when $n = 2k+1$ is odd}.
\end{cases}  
\end{equation}
\end{theorem}

We define two more polynomials:
\begin{eqnarray*}
R_n^{D,<}(t) & = & \sum_{\pi \in \DD_n, \pi_1<0} 
t^{\altr_D(\pi)}=\sum_{i=1}^n R_{n,i}^{D,<} t^i ,\\
R_n^{B-D,<}(t) & = &\sum_{\pi \in \BB_n-\DD_n, \pi_1<0} 
t^{\altr_D(\pi)}=\sum_{i=1}^n R_{n,i}^{B-D,<} t^i.
\end{eqnarray*}

\magenta{We denote  $\BB_n^> = \{ \pi \in \BB_n, \pi_1 > 0\}$.
Similarly, we define 
$\BB_n^< = \{ \pi \in \BB_n, \pi_1 < 0\}$, 
$\DD_n^> = \{ \pi \in \DD_n, \pi_1 > 0\}$ and
$\DD_n^< = \{ \pi \in \DD_n, \pi_1 < 0\}$.}

%\begin{remark}
%\label{rem:gao_sun_divis_by1_plus_t}
%Gao and Sun do not explicitly state the following, but their
%results imply that $R_n^{D,>}(t)$ and $R_n^{(B-D),>}(t)$
%are divisible by $(1+t)^{\nmhalf}$.  \blue{Perhaps
%give a line more of justification.}
%\end{remark}
\begin{remark}
\label{rem:firstnegativeequalsfirstpositive-dtype}
Gao and Sun do not explicitly state the following, but their
results imply that $R_n^{D,>}(t)$ and $R_n^{(B-D),>}(t)$
are divisible by $(1+t)^m$ where $m = \nmhalf$.  
Using the map $\sflip$, it is easy to see that 
$R_n^{D,<}(t)= R_n^{D,>}(t)$ when $n$ is even.  
When $n$ is odd, the same map gives us $R_n^{D,<}(t) =
R_n^{B-D,>}(t)$.
Combining with the first %Remark \ref{rem:gao_sun_divis_by1_plus_t}, 
we get that 
both $R_n^{D,<}(t)$ and  $R_n^{D,>}(t)$ 
are divisible by $(1+t)^m$.
Therefore, when $n$ is even,
the polynomial $R_n^D(t)$ 
is divisible by  $2(1+t)^m$.  When $n$ is
odd, $R_n^D(t)$ is divisible by $(1+t)^m.$
\end{remark}

\comment{
\begin{theorem}[Gao and Sun]
\label{thm:wilf-type-d-counterpart}
For positive integers $n$, the polynomial $R_n^D(t)$ is divisible by 
$(1 + t)^m$, where $m = \floor{(n -1)/2}$. 
\end{theorem}
}

The polynomials $R_n^D(t)$, $R_n^{B-D}(t)$, $R_n^{B,+}(t)$ and
$R_n^{B,-}(t)$ are related to each other as we show in Theorem
\ref{thm:altrunoverBplusequalsaltrunoverD}.  Towards
proving that, we first prove the following.

\begin{theorem}
\label{thm:Gao_sun_type_dfirstcoordinateall}
For positive integers $n$, we have
$$R_n^{D}(t) - R_n^{B-D}(t) =\begin{cases}
2t(1-t)(1-t^2)^{k-1} %\frac{n-2}{2} 
		& \text {when $n = 2k$ is even},\\
0 & \text {when $n$ is odd}.
\end{cases}  $$
  \end{theorem}
  
\begin{proof}
We use the map  
%$\sflip: \{\pi \in \BB_n: \pi_1>0\} \mapsto 
%\{\pi \in \BB_n: \pi_1>0\}.$
$\sflip: \BB_n^> \mapsto \BB_n^<.$
When $n$ is even, it is easy to see that
$\sflip: \{\pi \in \DD_n: \pi_1>0\} \mapsto \{\pi \in \DD_n: \pi_1<0\}$ 
and further the map is from $\BB_n^> -\DD_n^>$ to $\BB_n^< - \DD_n^<$.
That is, 
$\sflip: \{\pi \in \BB_n-\DD_n: \pi_1>0\} \mapsto 
\{\pi \in \BB_n-\DD_n: \pi_1<0\}$. 
Thus, we have
$$R_n^{D,<}(t) - R_n^{B-D,<}(t) =R_n^{D,>}(t) - R_n^{B-D,>}(t) =
t(1-t)(1-t^2)^\frac{n-2}{2}.$$
  	
When $n$ is odd, $\sflip$ is from $\DD_n^>$ to $\BB_n^< - \DD_n^<$.
and from $\BB_n^>-\DD_n^>$ to $\DD_n^<$.
Thus, we have
$$R_n^{D,<}(t) - R_n^{B-D,<}(t) = R_n^{B-D,>}(t)-R_n^{D,>}(t) =
-t(1-t^2)^\frac{n-2}{2}.$$
  	
Hence, we get
\begin{equation}
\label{eqn:difference_dfirstnegative_minus_bminusd_first_negative}
R_n^{D,<}(t) - R_n^{B-D,<}(t)  =\begin{cases}
  	t(1-t)(1-t^2)^{k-1} %\frac{n-2}{2} 
		& \text {when $n = 2k$ is even},\\
  	-t(1-t^2)^k %\frac{n-1}{2} 
		& \text {when $n = 2k+1$ is odd}.
  	\end{cases}  
  	\end{equation}
Summing up \eqref{eqn:difference_dfirstpositive_minus_bminusd_first_positive} 
and 
\eqref{eqn:difference_dfirstnegative_minus_bminusd_first_negative} 
completes the proof.
  \end{proof}

With this result, we immediately get the following. 
  
\begin{theorem}
\label{thm:altrunoverBplusequalsaltrunoverD}
For positive integers $n$, 
$$R_n^{B,+}(t)=R_n^D(t), \hspace{15 mm} R_n^{B,-}(t)=R_n^{B-D}(t).$$
\end{theorem} 
\begin{proof}
Since the sum of both quantities on the left hand sides and the right
	hand sides is $R_n^B(t)$, they are equal.  
From Corollary \ref{cor:univariate_signed_altrun_over_type_b} and 
Theorem \ref{thm:Gao_sun_type_dfirstcoordinateall}, their differences
are also equal.  This completes the proof.
\end{proof}

%From our results, we get Theorem \ref{thm:divisibiltyfortyped}.  
We move onto our next result, and give counterparts of Theorem
\ref{thm:divisibilty_of_alternatingrunpolynomial_by_highpower_oneplust}
to $D^{\pm}$ and $(B-D)^{\pm}$.  Though the first result in 
Theorem \ref{thm:divisibiltyfortyped} is due to Gao and Sun, 
our proof is different as it involves Theorem 
\ref{thm:Gao_sun_type_dfirstcoordinateall} and the type B polynomials.

\begin{theorem}
\label{thm:divisibiltyfortyped}
For positive integers $n$, the polynomials $R_n^{D}(t)$ and 
$R_n^{B-D}(t)$ are  divisible by $(1+t)^m$, where 
$m = \nmhalf$. Further, the polynomials  $R_n^{D, \pm}(t)$ 
and $R_n^{B-D, \pm}(t)$ are also divisible by $(1+t)^m$ 
where $m= \lfloor (n-1)/2\rfloor$. 
\end{theorem} 

\begin{proof}
Using Theorem \ref{thm:divisibiltyforplusminuspolynomialstypeB} and
Theorem \ref{thm:altrunoverBplusequalsaltrunoverD},
we get the result for $R_n^B(t)$ and $R_n^{B-D}(t)$. 
By Corollary
\ref{cor:univariate_signed_altrun_over_type_d}, we get 
the refinement for $R_n^{D \pm}(t)$.
%\red{Perhaps we can get a refinement for $R_n^{B-D,\pm}(t)$ as well.} 
%\magenta{
Moreover, as $R_n^{B-D,+}(t) = R_n^{B,-}(t) -  R_n^{D,-}(t)$ 
and  $R_n^{B-D,-}(t) = R_n^{B,+}(t) -  R_n^{D,+}(t)$, 
by Theorem 
\ref{thm:divisibiltyforplusminuspolynomialstypeB}, 
the polynomials $R_n^{B-D, \pm}(t)$ are also 
divisible by $(1+t)^m$. %}
\end{proof}

\subsection{Moment type identities in $\DD_n^{\pm}$ and $(\BB_n-\DD_n)^{\pm}$}
\label{subsec:type-d-moment-equalities}

Though Gao and Sun in \cite[Theorem 1.4]{Gaosunaltruntyped} 
do not explicitly mention the second part of Theorem
\ref{thm:sum_of_kth_powers_of_odd_equals_sum_of_kth_powers_signed_type_d}, 
it is easy to infer the second part from 
Lemma \ref{lem:chow-ma-powers-alt} and Theorem 
\ref{thm:sum_of_kth_powers_of_odd_equals_sum_of_kth_powers_signed_type_d}.
The following is thus known.

\begin{theorem}[Gao and Sun]
\label{thm:sum_of_kth_powers_of_odd_equals_sum_of_kth_powers_signed_type_d}
For $n \geq 2k+3$, 
\begin{eqnarray*}
1^kR_{n,1}^{D,>}+3^kR_{n,3}^{D,>}+5^kR_{n,5}^{D,>}+\cdots  
& = & 
2^kR_{n,2}^{D,>}+4^kR_{n,4}^{D,>}+6^kR_{n,6}^{D,>}+\cdots, \\
1^kR_{n,1}^{B-D,>}+3^kR_{n,3}^{B-D,>}+5^kR_{n,5}^{B-D,>}+\cdots  
& = &
2^kR_{n,2}^{B-D,>}+4^kR_{n,4}^{B-D,>}+6^kR_{n,6}^{B-D,>}+\cdots.
\end{eqnarray*}
\end{theorem}

From Lemma \ref{lem:chow-ma-powers-alt} and 
Remark \ref{rem:firstnegativeequalsfirstpositive-dtype}
the following slightly
weaker identity follows.

\begin{corollary}
\label{thm:moment_identities_type_d}
For positive integers $n$ and positive integers $k$ with $n \geq 2k+3$, 
\begin{eqnarray*}
\label{eqn:sum_of_kth_powers_of_odd_equals_sum_of_kth_powers_of_even_type_d}
1^kR_{n,1}^D+3^kR_{n,3}^D+5^kR_{n,5}^D+\cdots 
& = & 
2^kR_{n,2}^D+4^kR_{n,4}^D+6^kR_{n,6}^D+\cdots, 
\end{eqnarray*}
\end{corollary}

From Lemma 
\ref{lem:chow-ma-powers-alt} and Theorem   
\ref{thm:divisibiltyfortyped}, 
the following refinement of Corollary
\ref{thm:moment_identities_type_d} is straightforward.

\begin{theorem}
\label{thm:moment-ids-typed-pm}
For positive integers $n$ and positive integers $k$ with $n \geq 2k+3$, 
\begin{equation*}
\label{eqn:sum_of_kth_powers_of_odd_equals_sum_of_kth_powers_of_even_sign_type_d_plus}
1^kR_{n,1}^{D, \pm }+3^kR_{n,3}^{D, \pm}+5^kR_{n,5}^{D, \pm}+\cdots  
=  
2^kR_{n,2}^{D, \pm}+4^kR_{n,4}^{D, \pm}+6^kR_{n,6}^{D, \pm}+\cdots 
\end{equation*}
\end{theorem}

%From Lemma \ref{lem:chow-ma-powers-alt} and Theorem 

A signed refinement of Theorem 
\ref{thm:sum_of_kth_powers_of_odd_equals_sum_of_kth_powers_signed_type_d}
is also true.  As mentioned earlier, we refer the reader to 
\cite[Theorem 18]{dey-siva-grp-action-bona-wilf} 
for this.

\subsection{Alternating permutations in $\BB_n^{\pm}$, $\DD_n^{\pm}$}
\label{subsec:alt_perms-bd}  
Type B analogues of Theorem \ref{thm:egf_for_alt_perm_symmetricgroup} 
are known.  A permutation $\pi = \pi_1, \pi_2, \cdots, \pi_n \in \BB_n$ 
is called alternating 
if $\pi_1 > \pi_2 < \pi_3 > \pi_4 < \cdots.$  That is, 
$\pi_i < \pi_{i+1}$ for $i$ 
even, and $\pi_i > \pi_{i+1}$ for $i$ odd. Let, $\Alt_n^B$ be the set of 
alternating permutations in $\BB_n$ and let $E_n^B= |\Alt_n^B|$.  
The following counterpart of Andr{\'e}'s result is due to 
Steingr{\'i}msson \cite[Theorem 34]{steingrimsson-indexed-perms}.
\begin{theorem}[Steingr{\'i}msson]
\label{thm:egf_for_alt_perm_hyperoctahedralgroup}
For positive integers $n,$ we have
\begin{equation}
\label{eqn:egf_for_alt_perm_hyperoctahedralgroup}
  	\sum_{n=0}^{\infty} {E_n^B}\frac{x^n}{n!}	
  	=  \sec 2x+ \tan 2x. 
  	\end{equation}
  \end{theorem}

In the type A case, we had done enumeration with respect
to the type (that is, ascent/descent) of both the first 
and the last pair of elements.
As we do not enumerate based on these two pairs for the type B and 
type D cases, we need to do a little more work. 
Define $E_n^{B,+}= |\Alt^B_n \cap \BB_n^+|$ and $E_n^{B,-}= 
|\Alt^B_n \cap \BB_n^-|.$  Define $\Alt^D_n= \Alt^B_n \cap \DD_n$ 
and  $\Alt^{B-D}_n= \Alt^B_n \cap (\BB_n- \DD_n).$ 
Let $E_n^{D}= |\Alt^D_n|$ and $E_n^{B-D}= |\Alt^{B-D}_n|.$ 
We start with the following Lemma.
 
\begin{lemma}
\label{lem:difference_of_bplus_bminus_alternating_perms_typeb}
For positive integers $n$, we have  $E_n^{B,+}= E_n^{B,-}= E_n^{D}
=E_n^{B-D}=  E_n^{B}/2$. 
\end{lemma}
 
\begin{proof}
Let, $\pi= \pi_1, \pi_2, \dots, \pi_j=x,\dots, \pi_n \in \Alt^B_n \cap \BB_n^+$ 
where $|x|= 1$.% or $x= -1$. 
Define $f: \Alt^B_n \cap \BB_n^+ \mapsto \Alt^B_n \cap \BB_n^- $ by
$$f( \pi_1, \pi_2, \dots, \pi_j=x,\dots, \pi_n)  = \pi_1, \pi_2, \dots, 
\pi_j= \overline{x},\dots, \pi_n.$$ That is, $f$ flips the sign of the 
smallest element in absolute value.  Clearly, $f$ is a bijection.
When $\pi$ is an alternating permutation, it is easy to note that  
$f(\pi)$ is also an alternating permutation.
Moreover, $f(\pi) \in \BB_n^+$ iff $\pi \in \BB_n^-$.  %have $\inv_B$ of different parity. 
Therefore, $E_n^{B,+}= E_n^{B,-}= E_n^{B}/2$.  We need an identical
map but now defined on the set $\Alt^D_n$.  Thus, define 
$g: \Alt^D_n  \mapsto \Alt^{B-D}_n $
by 
$$g( \pi_1, \pi_2, \dots, \pi_j=x,\dots, \pi_n)  = \pi_1, \pi_2, \dots, \pi_j= \overline{x},\dots, \pi_n.$$
Clearly, $g$ is also a bijection and further $g(\pi)$
is an alternating permutation. With these maps, the proof is complete. 
\end{proof}
 
An immediate consequence of Lemma
\ref{lem:difference_of_bplus_bminus_alternating_perms_typeb}
is the following.  Here, we assume $E_0^{B,+}=1, E_0^{B,-}=0$ and 
$E_0^{D}=1, E_0^{B-D}=0$.
 
 \begin{theorem}
 \label{thm:egf_for_alt_perm_hyperoctahedralgroupplusandtyped}
 We have the following exponential generating functions.
 \begin{eqnarray}
\label{eqn:egf_for_alt_perm_hyperoctahedralgroupplusminus}
\sum_{n=0}^{\infty} {E_n^{B, \pm}}\frac{x^n}{n!}	
& =  &(\sec 2x+ \tan 2x \pm 1 )/2,   \\%\hspace{5 mm}
%\sum_{n=0}^{\infty} {E_n^{B-}}\frac{x^n}{n!}	
%=  (\sec 2x+ \tan 2x -1)/2, \\
\label{eqn:egf_for_alt_perm_demihyperoctahedralgroup}
\sum_{n=0}^{\infty} {E_n^{D}}\frac{x^n}{n!}	
& =  & (\sec 2x+ \tan 2x +1 )/2,\hspace{3 mm}
\sum_{n=0}^{\infty} {E_n^{B-D}}\frac{x^n}{n!} 	
=  (\sec 2x+ \tan 2x -1 )/2.
 \end{eqnarray}
 \end{theorem}
 
\begin{proof}
We have
\begin{eqnarray*}
\sum_{n=0}^{\infty} {E_n^{B \pm}}\frac{x^n}{n!}	
& = & 
\frac{1}{2}\sum_{n=0}^{\infty} {E_n^B}\frac{x^n}{n!} \pm \frac{1}{2} 
\sum_{n=0}^{\infty}\left( E_n^{B,+} - E_N^{B,-} \right) \frac{x^n}{n!} \nonumber \\
&  = & \frac{1}{2}\big( \sec 2x+ \tan 2x \big) \pm\frac{1}{2}  = \frac{1}{2}\big(\sec 2x+\tan 2x \pm 1\big). 
 	\end{eqnarray*}
We have used Theorem \ref{thm:egf_for_alt_perm_hyperoctahedralgroup} 
and Lemma
\ref{lem:difference_of_bplus_bminus_alternating_perms_typeb} in the 
second line. The proof of 
\eqref{eqn:egf_for_alt_perm_demihyperoctahedralgroup} follows from 
Lemma 
\ref{lem:difference_of_bplus_bminus_alternating_perms_typeb}.
 \end{proof}
 
We now move on to alternating permutations in the positive elements 
of Type D Coxeter Groups. 
Define  $\Alt^{D,+}_n= \Alt^D_n \cap \DD_n^{+}$  and 
$\Alt^{D,-}_n = \Alt^D_n \cap \DD_n^{-}$  and let
$E_n^{D,+}= |\Alt^{D,+}_n|$ and $E_n^{D,-}= |\Alt^{D,-}_n|.$  Similarly,
define  $\Alt^{B-D,+}_n= \Alt^{B-D}_n \cap (\BB_n-\DD_n)^{+}$  and 
$\Alt^{B-D,-}_n = \Alt^{B-D}_n \cap (\BB_n - \DD_n)^{-}$  and let
$E_n^{B-D,+}= |\Alt^{B-D,+}_n|$ and $E_n^{B-D,-}= |\Alt^{B-D,-}_n|.$
 
\begin{lemma}
\label{lem:difference_of_bplus_bminus_alternating_perms_typed}
For positive integers $n \geq 2$, we have  $E_n^{D,+}= E_n^{D,-}=  \frac{1}{2} E_n^{D}$.  
We also have $E_n^{B-D,+}= E_n^{B-D,-}=  \frac{1}{2}  E_n^{B-D}$.  
\end{lemma}
  
\begin{proof}
Let, $\pi= \pi_1, \pi_2, \dots, \pi_i=x,\dots,\pi_j=y,\dots, \pi_n 
\in \Alt^{D,+}_n $ 
where $|x|= 1$ and $|y|= 2$. 
Define $h: \Alt^{D,+}_n  \mapsto \Alt^{D,-}_n $ by
$$h(\pi_1, \pi_2, \dots, \pi_i=x,\dots,\pi_j=y,\dots, \pi_n )  = \pi_1, \pi_2, \dots, 
\pi_i=\overline{y},\dots,\pi_j=\overline{x},\dots, \pi_n$$
The map $h$ is clearly a bijection.
It is easy to see that $h(\pi)$ is an alternating permutation whenever
$\pi$ is.  It is also clear that $h(\pi) $ and $\pi$ have opposite parity of 
$\inv_D$. The same map $h$ applied on $\pi \in \BB_n-\DD_n$ 
flips parity of $\inv_D$ and shows that 
$E_n^{B-D,+} = E_n^{B-D,-} = \frac{1}{2} E_n^{B-D}$.

%of different parity. 
The proof is complete.
\end{proof}
  
Using Lemma
\ref{lem:difference_of_bplus_bminus_alternating_perms_typed}, 
we have the following result. The proof is so similar 
to the proof of Theorem \ref{thm:egf_for_alt_perm_hyperoctahedralgroupplusandtyped}
and so we omit a proof and merely state our result. Here, 
we assume $E_0^{D,+}=1$ and $E_0^{D,-}=0$.  We also 
assume that $E_n^{B-D,+} = E_n^{B-D,-} = 0$.
 
\begin{theorem}
\label{thm:egf_for_alt_perm_demihyperoctahedralgroupplusminus}
We have
\begin{eqnarray*}
\label{eqn:egf_for_alt_perm_demihyperoctahedralgroupplus}
\sum_{n=0}^{\infty} {E_n^{D, \pm}}\frac{x^n}{n!}	
&  = &   \big(\sec 2x+ \tan 2x\pm2x + \delta^{\pm} \big)/4, 
\mbox{where $\delta^+ = 3$ and $\delta^- = -1$,} \\
\sum_{n=0}^{\infty} {E_n^{B-D, \pm}}\frac{x^n}{n!}	
&  = &   \big(\sec 2x+ \tan 2x -1 \pm x  \big)/2.
%\\
%  	\label{eqn:egf_for_alt_perm_demihyperoctahedralgroupminus}
%  	\sum_{n=0}^{\infty} {E_n^{D -}}\frac{x^n}{n!}	
%    &  = & \big(\sec 2x+ \tan 2x-2x-1\big)/4.
\end{eqnarray*}
\end{theorem}

\subsection{Springer permutations in $\BB_n^{\pm}$ and $\DD_n^{\pm}$ }

A {\it snake} of type $B$ is a signed permutation $\pi_1,\pi_2,\cdots,\pi_n \in \BB_n$ 
such that $0<\pi_1>\pi_2< \pi_3 > \cdots$. Let, $\Snake_n^B$ be the set of 
snakes in $\BB_n$ and let $S_n^B= |\Snake_n^B|$. The numbers $S_n^B$ are also 
called Springer numbers. 
Springer in \cite{springer-remarks-combin}
%\magenta{ have to cite EJC 2016 shi mei ma paper } 
derived the following generating function:
$$S(x)=\sum_{n\geq 0}S^B_n\frac{x^n}{n!}=\frac{1}{\cos x-\sin x}.$$

Define $\Snake^{B,+}_n=\Snake^B_n \cap \BB_n^+$ to be the set of 
positive type B snakes and let $S_n^{B,+}= |\Snake^{B,+}_n|.$ 
Similarly, let  $\Snake^{B,-}_n= \Snake^B_n \cap \BB_n^-$ and 
let $S_n^{B,-}= |\Snake^{B,-}_n|.$   Further, 
define $\Snake^D_n= \Snake^B_n \cap \DD_n$ and 
$\Snake^{B-D}_n= \Snake^B_n \cap (\BB_n- \DD_n).$  Let 
$S_n^{D}= |\Snake^D_n|$ and $S_n^{B-D}= |\Snake^{B-D}_n|.$
Gao and Sun in  \cite[Corollary 2.5]{Gaosunaltruntyped}
proved the following.

\begin{theorem}[Gao and Sun]
	\label{thm:dif_snakes_d_bminusd}
	For positive integers $n$, we have
	\begin{eqnarray*}
		S_n^{D} - S_n^{B-D}& = & \begin{cases}
			1 & \text {if $n=0,1$ (mod 4)},\\
			-1 & \text {if $n=2,3$  (mod 4)}.
		\end{cases}
	\end{eqnarray*}
\end{theorem}	

We start by proving the following lemma.

\begin{lemma}
\label{lem:dif_snakes_oddeven_b}
For positive integers $n$, we have
\begin{eqnarray*}
S_n^{B,+} - S_n^{B,-}& = & \begin{cases}
1 & \text {if $n=0,1$ (mod 4)},\\
-1 & \text {if $n=2,3$ (mod 4)}.
\end{cases}
\end{eqnarray*}
\end{lemma}	

\begin{proof}
Let $\id = 1,2,3,4,\dots,n-1,n \in \BB_n^>$ be the identity permutation. Consider the 
set $\BB_n^>(\id) = \{\pi \in \BB_n^> : \pi_1=1 \mbox{ and }  |\pi_i|=i \mbox{ for } i>1 \}$. 
For positive integers $n$, clearly, $|\BB_n^>(\id)| = 2^{n-1}$.
Let $\pi= \pi_1, \pi_2, \dots, \pi_n \in \Snake_n^{B,+} \cap (\BB_n - \BB_n^>(\id))$. 
As $\pi \notin \BB_n^>(\id)$, there exists $j<n$ such that $|\pi_j| \neq j$. 
Let $k$ be the smallest index such that $|\pi_k| \neq k$.  Let $r > k$ be 
the index with $|\pi_r|=k$.  Note that $r \geq 2$ and $r < n$.  As $k$ is the smallest non-fixed point 
in absolute value, we have $|\pi_i|=i$ for $1 \leq i \leq k-1$.  As $2 \leq r < n$, both 
$\pi_{r-1}$ and $\pi_{r+1}$ exist and we have $|\pi_{r-1}| > k$ and $|\pi_{r+1}| > k$. Define 
$f: \Snake_n^{B,+} \cap (\BB_n - \BB_n^>(\id) ) \mapsto \Snake_n^{B,-} \cap (\BB_n - \BB_n^>(\id) )$ 
by $$f( \pi_1, \pi_2, \dots, \pi_r=k,\dots, \pi_n)  = \pi_1, \pi_2, \dots, 
\pi_r= \overline{k},\dots, \pi_n.$$ 
When $\pi$ is snake, it is easy to observe that $f(\pi)$ is also a snake. 
Further, as $\pi \in \BB_n^+$, we have $f(\pi) \in \BB_n^-.$  Thus $f$ is a well defined map and
clearly, it is a bijection. Therefore, we have 
$|\Snake_n^{B,+} \cap  ( \BB_n - \BB_n^>(\id)| = |\Snake_n^{B,-} \cap  ( \BB_n - \BB_n^>(\id)|. $ Thus, 
$$S_n^{B,+} - S_n^{B,-}= |\Snake_n^{B,+} \cap \BB_n^>(\id)|-  |\Snake_n^{B,-} \cap \BB_n^>(\id)|. $$ 

We now count the number of snakes in the set $\BB_n^>(\id)$. 
Let $\pi$ be a snake in $\BB_n^>(\id)$.  We have $\pi_1=1$. 
As $\pi$ is a snake in $\BB_n^>(\id)$, we have $\pi_2 < \pi_1$.  As $\pi_2 \in \{2, \ol{2}\}$, 
$\pi_2= \ol{2}$. Similarly, 
%$\pi_3 > \pi_2$ and $\pi_3 \in \{3, \ol{3}\}$. Thus, 
$\pi_3= 3$.  Continuing in this way, $\pi$ must have $\pi_r=r$ when $r$ is odd and 
$\pi_r=\ol{r}$ when $r$ is even.  Thus, for any positive integer $n$, the set 
$\BB_n^>(\id)$ has only 
$1$ snake, namely $ \pi= 1, \ol{2},3,\ol{4},5,\ol{6},\dots$. The permutation 
$\pi  \in \BB_n^+$ 
if and only if $\pi$ has even number of negative entries which happens if and only if $n$ 
is $0,1$ (mod 4). The proof is complete. 
\end{proof}

With initial values $S_0^{B,+}=1$, $S_0^{B,-}=0$, $S_0^{D}=1$ and $S_0^{B-D}=0$,
we get the following consequence of Lemma \ref{lem:dif_snakes_oddeven_b}. 

\begin{theorem}
\label{thm:egfforsnakesb}
We have the following exponential generating functions.
\begin{eqnarray}
\label{eqn:egf_snakes_typeb_plus}
\sum_{n=0}^{\infty} {S_n^{B,+}} \frac{x^n}{n!}
& =  & \frac{\cos ^2 x }{\cos x - \sin x},   \hspace{25 mm}
\sum_{n=0}^{\infty} {S_n^{B,-}} \frac{x^n}{n!}
 =   \frac{\sin ^2 x }{\cos x - \sin x}, \\%\hspace{5 mm}
%\sum_{n=0}^{\infty} {E_n^{B-}}\frac{x^n}{n!}	
%=  (\sec 2x+ \tan 2x -1)/2, \\
\label{eqn:egf_snakes_typed}
\sum_{n=0}^{\infty} {S_n^{D}}\frac{x^n}{n!}	
& =  & \frac{\cos ^2 x }{\cos x - \sin x},\hspace{25 mm}
\sum_{n=0}^{\infty} {S_n^{B-D}}\frac{x^n}{n!} 	
=  \frac{\sin ^2 x }{\cos x - \sin x}.
\end{eqnarray}
\end{theorem}

\begin{proof}
We first consider \eqref{eqn:egf_snakes_typeb_plus}.
We have
\begin{eqnarray*}
\sum_{n=0}^{\infty} {S_n^{B, \pm}}\frac{x^n}{n!}	& = & 
\frac{1}{2}\sum_{n=0}^{\infty} {S_n^B}\frac{x^n}{n!} \pm 
\frac{1}{2} \sum_{n=0}^{\infty}\big( S_n^{B+} - S_N^{B-} \big) \frac{x^n}{n!} \nonumber \\
		&  = & \frac{1}{2}\left[ \frac {1}{\cos x - \sin x} \pm (1 + x - \frac{x^2}{2!}-\frac{x^3}{3!}+\dots) \right] \\
		& = & \frac {1}{2}\left[ \frac {1}{\cos x - \sin x} \pm (\cos x+ \sin x) \right]. 
	\end{eqnarray*}
We have used Lemma \ref{lem:dif_snakes_oddeven_b} in the second line. This completes the proof 
of \eqref{eqn:egf_snakes_typeb_plus}. 
Theorem \ref{thm:dif_snakes_d_bminusd} and Lemma \ref{lem:dif_snakes_oddeven_b}
imply that $S_n^{B,+} = S_n^D$ and $S_n^{B,-} = S_n^{B-D}$.
This remark implies \eqref{eqn:egf_snakes_typed}, completing the proof.
\end{proof}

\vspace{2 mm}
	
We move to Snakes of type D.  
Define $\Snake_n^{D,+}= \Snake^D_n \cap \DD_n^+$ and $\Snake_n^{D,-}= \Snake^D_n \cap \DD_n^-$ and 
let $S_n^{D,+}= |\Snake_n^{D,+}|$, $S_n^{D,-}= |\Snake_n^{D,-}|$.  
Similarly, define $\Snake_n^{B-D,+} = 
\Snake_n^{B-D} \cap (\BB_n - \DD_n)^+$ and $\Snake_n^{B-D,-} = 
\Snake_n^{B-D} \cap (\BB_n - \DD_n)^-$ and let 
$S_n^{B-D,+}= |\Snake_n^{B-D,+}| $ and  $S_n^{B-D,-}= |\Snake_n^{B-D,-}|$. 	
We partition snakes in $D_n$ into the following subsets:
\begin{enumerate} 
	\item $L_n^1= \{ \pi \in \Snake^{D}_n :%\pi'' \in \BB_{n-2} \mbox{ and }  
|\pos_{\pm n}(\pi) - \pos_{\pm (n-1)}(\pi)| > 1  \}.$
	\item $L_n^2= \{ \pi \in \Snake^{D}_n : %\pi'' \in \BB_{n-2} \mbox{ with }  
|\pos_{\pm n}(\pi) - \pos_{\pm (n-1)}(\pi)| =1  \mbox{ and }  \pi_n \notin \{ \pm (n-1), \pm  n  \} \}.$
	\item $L_n^3= \{ \pi \in \Snake^{D}_n : %\pi'' \in \BB_{n-2} \mbox{ with }  
|\pos_{\pm n}(\pi) - \pos_{\pm (n-1)}(\pi)| =1  \mbox{ and }  \pi_n \in \{ \pm (n-1), \pm  n \}, \sign(\pi_{n-1}) \neq \sign(\pi_{n})  \}.$
	\item $L_n^4= \{ \pi \in \Snake^{D}_n :%\pi'' \in \BB_{n-2} \mbox{ with }  
|\pos_{\pm n}(\pi) - \pos_{\pm (n-1)}(\pi)| =1  \mbox{ and }  \pi_n \in \{ \pm (n-1), \pm  n \}, \sign(\pi_{n-1}) = \sign(\pi_{n})  \}.$
\end{enumerate}
We next show that the number of snakes in $L_n^i \cap \DD_n^{+} $ 
and $L_n^i \cap \DD_n^{-} $ are equal when $1 \leq i \leq 3$. 

\begin{lemma}
\label{lem:three_equal_plus_minus}
When $n \geq 3$, we have 
\begin{equation*}
| L_n^1 \cap \DD_n^{+}|  =  | L_n^1  \cap \DD_n^{-}|, \hspace{3 mm}
| L_n^2 \cap \DD_n^{+}|  =  | L_n^2  \cap \DD_n^{-}|  \mbox{ and }
| L_n^3 \cap \DD_n^{+}|  =  | L_n^3  \cap \DD_n^{-}|.
\end{equation*}
\end{lemma}
\begin{proof}
%We show the three parts separately. 
For the first part, let 
$\pi = \pi_1, \dots ,\pi_{i-1},\pi_i= x, \pi_{i+1}, \dots, 
\pi_j=y, \dots, \pi_n \in L_n^1 \cap \DD_n^{+}$, where 
$\{|x|,|y|\} =\{n-1,n\}$. 
Define $g: L_n^1 \cap \DD_n^{+} \mapsto L_n^1 \cap \DD_n^{-} $ by 
$$g(\pi)= \begin{cases}
\pi_1, \dots, ,\pi_{i-1},\pi_i= y, \pi_{i+1}, \dots, \pi_j=x, \dots, \pi_n & \text {if $\sign(x)= \sign(y)$}, \\
\pi_1, \dots, ,\pi_{i-1},\pi_i= \overline{y}, \pi_{i+1}, \dots, \pi_j=\overline{x}, \dots, \pi_n & \text {if $\sign(x) \neq  \sign(y)$}.
\end{cases}$$ 
When $\pi$ is a snake, it is easy to see that $g(\pi)$ is also a 
snake. The map $g$ clearly flips the parity of $\inv_D$. 
%\blue{As done in the proof of Lemma 30...}
This completes the proof of the first part.

For the second part, let 
$\pi = \pi_1, \dots ,\pi_{i-1},\pi_i= x, \pi_{i+1}=y, \dots, 
\pi_n \in L_n^2 \cap \DD_n^{+}$, where $\{|x|,|y|\} =\{n-1,n\}$. 
Clearly, $n,n-1$ are the largest two elements in absolute value and they
are not at the end of $\pi$.  As $\pi$ is a snake, $x$ and $y$ 
are must be of opposite sign. 
We consider the map 
$f:L_n^2 \cap \DD_n^{+} \mapsto L_n^2 \cap \DD_n^{-}$ as follows.
$$f(\pi)= 
\pi_1, \dots, ,\pi_{i-1},\pi_i= \ol{y}, \pi_{i+1}=\ol{x}, \dots, \pi_j=x, \dots, \pi_n.$$
As $\pi$ is snake, clearly $f(\pi)$ is also a snake. Moreover, $f$ 
flips the parity of $\inv_D$. This completes the proof of the second
part. 

For the final part, the same map $f$, that we used 
in the second part works here as well. This completes the 
proof.
\end{proof}

\begin{theorem}
	\label{thm:dif_snakes_oddeven_d}
	For positive integers $n \geq 3$, we have
	\begin{eqnarray}
	\label{eqn:snakesjumpby2}
S_{n}^{D,+} - S_{n}^{D,-} & = & S_{n-2}^{D,-} - S_{n-2}^{D,+}, \\
		\label{eqn:snakesbminusdjumpby2}	
		 S_{n}^{B-D,+} - S_{n}^{B-D,-} & = & S_{n-2}^{B-D,-} - S_{n-2}^{B-D,+}.
\end{eqnarray}
With initial values, we have 
	\begin{eqnarray}
	\label{eqn:snakedplusminus}
		S_n^{D,+} - S_n^{D,-}& = & \begin{cases}
			1 & \text {if $n=0,1$ (mod 4)},\\
			-1 & \text {if $n=2,3$ (mod 4)}.
		\end{cases} \\
	\label{eqn:snakebminusdplusminus}	
		S_n^{B-D,+} - S_n^{B-D,-} & = & 0.
	\end{eqnarray}
\end{theorem}	

\begin{proof}
We first consider \eqref{eqn:snakesjumpby2}. 
By Lemma \ref{lem:three_equal_plus_minus}, we have 
\begin{equation}
\label{eqn:snakesjumpby2s4}
S_{n}^{D,+} - S_{n}^{D,-} = | L_n^4 \cap \DD_n^{+}| -  | L_n^4 \cap \DD_n^{-}|
\end{equation}

Let $\pi \in L_n^4$. Then, as $\pi$ is a snake in $\DD_n$ and 
$\sign(x) =  \sign(y)$, we get that $\pi''$ is also a snake in $\DD_{n-2}$. We now consider two cases. 

\vspace{2 mm}

{\bf When n is even:}

\vspace{2 mm}

We denote $\pi''$ as $\sigma$ and let 
$\sigma= \sigma_1, \sigma_2, \dots, \sigma _{n-2} \in 
\Snake_{n-2}^{D,+}.$ %\times \DD_{n-2}^{+}$.  \blue{What is this $\times$
%notation?  Is ``intersection" meant? Further, why has the
%notation changed from $\pi''$ to $\sigma$?}
The only way to get $\pi \in L_n^4 \cap \DD_n$ is by 
adding the string $n, n-1$ at the end. That is, by getting 
$\pi= \sigma_1, \sigma_2, \dots, \sigma _{n-2}, n, n-1$. 
We clearly have $\inv_D(\pi) = \inv_D(\pi'') + 1$.
As $\sigma \in \DD_{n-2}^{+}$, we have 
$\pi \in \DD_n^{-}$. Thus, all permutations in 
$\Snake_{n-2}^{D,+}$ %\times \DD_{n-2}^{+}$ \blue{Again the $\times$ is 
%confusing.}
contribute exactly one permutation to $L_n^4 \cap \DD_n^{-}$. 
In an identical manner one can show that  all the permutations in 
$\Snake_{n-2}^{D,-}$ %\times \DD_{n-2}^{-}$ 
contribute exactly one permutation to $S_n^4 \cap \DD_n^{+}$. 
Thus, we have 
\begin{equation} 
\label{eqn:jumpby2neven}
|L_n^4 \cap \DD_n^{+}| -  | L_n^4 \cap \DD_n^{-}| = S_{n-2}^{D,-} - S_{n-2}^{D,+}.
\end{equation}

{\bf When n is odd:}

\vspace{2 mm}

We again denote $\pi''$ as $\sigma$ and let 
$\sigma= \sigma_1, \sigma_2, \dots, \sigma _{n-2} \in \Snake_{n-2}^{D,-}.$
%\times \DD_{n-2}^{+}$.  
Clearly, one can get $\pi \in L_n^4 \cap \DD_n$ only by 
adding  $\ol{n}, \ol{n-1}$ at the end. That is 
$\pi= \sigma_1, \sigma_2, \dots, \sigma _{n-2}, \ol{n}, \ol{n-1}$. 
Again, it is clear that the parity of $\inv_D(\pi)$ and 
$\inv_D(\sigma)$ are different.
As $\sigma \in \DD_{n-2}^{+}$, we have
$\pi \in \DD_{n}^{-}$. Thus, all permutations in 
$\Snake_{n-2}^{D,+}$ give rise to exactly one
permutation to $L_n^4 \cap \DD_n^{-}$.  As before,  
all permutations in $\Snake_{n-2}^{D,-}$ %\times \DD_{n-2}^{-}$ 
give rise to exactly one permutation in 
$L_n^4 \cap \DD_n^{+}$. Thus, we have 
\begin{equation} 
\label{eqn:jumpby2nodd}
|L_n^4 \cap \DD_n^{+}| -  | L_n^4 \cap \DD_n^{-}| = S_{n-2}^{D,-} - S_{n-2}^{D,+}.
\end{equation}

Thus, by \eqref{eqn:jumpby2neven} and \eqref{eqn:jumpby2nodd}, we get that for any positive integer $n$,
\begin{equation}
\label{eqn:jumpby2nbothoddeven}
| L_n^4 \cap \DD_n^{+}| -  | L_n^4 \cap \DD_n^{-}| = S_{n-2}^{D,-} - S_{n-2}^{D,+}.
\end{equation}
By \eqref{eqn:snakesjumpby2s4} and \eqref{eqn:jumpby2nbothoddeven}, the proof of \eqref{eqn:snakesjumpby2} is complete. In an identical manner, one can prove \eqref{eqn:snakesbminusdjumpby2}.
%\end{proof}
%\begin{theorem}
%\label{thm:signed-d-snakes}
%	When $n \geq 3$ \magenta{$n \geq 1$ itself?}, 
%
%\end{theorem}
%\begin{proof}
The following base cases are easy to see.  When $n=1,$ 
 we have $S_{1}^{D,+}=1$, $S_{1}^{D,-}=0$, $S_{1}^{B-D,+}=0$ and $S_{1}^{B-D,-}=0$. Thus, 
 $S_{1}^{D,+}-S_{1}^{D,-}=1$ and
  $S_{1}^{B-D,+}-S_{1}^{B-D,-}=0$.  
 When $n=2$, we have $S_{2}^{D,+}=0$, $S_{2}^{D,-}=1$, $S_{2}^{B-D,+}=1$ and $S_{2}^{B-D,-}=1$. 
 Thus, 
 $S_{2}^{D,+}-S_{2}^{D,-}=-1$ and
 $S_{2}^{B-D,+}-S_{2}^{B-D,-}=0$.  
Using the recurrences \eqref{eqn:snakesjumpby2} and \eqref{eqn:snakesbminusdjumpby2}, the proof is complete.  
\end{proof}

From Lemma \ref{lem:dif_snakes_oddeven_b}, we have the following. The proof of Theorem \ref{thm:egfforsnakesd}
is identical to the proof of Theorem \ref{thm:egfforsnakesb} and therefore we omit this. Here we assume
$S_0^{D,+}=1$, $S_0^{D,-}=0$, $S_0^{B-D,+}=0$ and $S_0^{B-D,-}=0$.

\begin{theorem}
	\label{thm:egfforsnakesd}
	We have the following exponential generating functions.
	\begin{eqnarray}
	\label{eqn:egf_snakes_typed_plus}
	\sum_{n=0}^{\infty} {S_n^{D,+}} \frac{x^n}{n!}
	& =  & \frac{2\cos ^2 x - \sin ^2 x }{2(\cos x - \sin x)},   \hspace{25 mm}
	\sum_{n=0}^{\infty} {S_n^{D,-}} \frac{x^n}{n!}
	=   \frac{\sin ^2 x }{2(\cos x - \sin x)}, \\%\hspace{5 mm}
	%\sum_{n=0}^{\infty} {E_n^{B-}}\frac{x^n}{n!}	
	%=  (\sec 2x+ \tan 2x -1)/2, \\
	\label{eqn:egf_snakes_typebminusd}
	\sum_{n=0}^{\infty} {S_n^{B-D, \pm }}\frac{x^n}{n!}	
	& =  & \frac{\sin ^2 x }{2(\cos x - \sin x)}.
%	\hspace{25 mm}
%	\sum_{n=0}^{\infty} {S_n^{B-D}}\frac{x^n}{n!} 	
%	=  \frac{\sin ^2 x }{2(\cos x - \sin x)}
	\end{eqnarray}
\end{theorem}

\bibliographystyle{acm}
%\bibliography{main-sp3}
\bibliography{main}
\end{document}